\documentclass[12pt]{amsart}
\usepackage[scr=boondoxo,scrscaled=1.05]{mathalfa}
\usepackage{fullpage,amssymb}
\usepackage{graphicx}
\usepackage[export]{adjustbox}
\usepackage{amsrefs}
\usepackage{xy}
\xyoption{all}

\newtheorem{theorem}{Theorem}[section]
\newtheorem{lemma}[theorem]{Lemma}
\newtheorem{corollary}[theorem]{Corollary}
\newtheorem{proposition}[theorem]{Proposition}
\newtheorem*{notation*}{Notation}
\newtheorem*{p*}{Proposition~\ref{h.s.o.p}}
\theoremstyle{definition}
\newtheorem{definition}[theorem]{Definition}
\newtheorem{example}[theorem]{Example}
\newtheorem{remark}[theorem]{Remark}

\newtheorem{thm}[theorem]{Theorem}
\newtheorem{lem}[theorem]{Lemma}
\newtheorem{prop}[theorem]{Proposition}

\newcommand{\Mat}{\operatorname{Mat}}

\newcommand{\MM}{\mathscr{M}}
\newcommand{\Tr}{\operatorname{Tr}}
\newcommand{\M}{{\operatorname{Mat}}}

\newcommand{\C}{{\mathbb C}}
\newcommand{\N}{{\mathbb Z}_{\geq 0}}
\newcommand{\Z}{{\mathbb Z}}

\newcommand{\K}{K}

\newcommand{\sgn}{\operatorname{sgn}}
\newcommand{\trace}{\operatorname{tr}}
\newcommand{\Hom}{\operatorname{Hom}}

\newcommand{\GL}{\operatorname{GL}}

\newcommand{\V}{\mathcal{V}}

\newcommand{\ZZ}{\mathcal{Z}}
\newcommand{\GLn}{{\mathscr A}(n)}
\newcommand{\BBn}{{\mathscr W}(n)}
\newcommand{\id} {\operatorname{id}}
\newcommand{\ten} {10}
\newcommand{\eleven} {11}
\newcommand{\PR} {\mathcal{P}}
\newcommand{\W} {\mathcal{W}}
\newcommand{\End}{\operatorname{End}}
\newcommand{\type}{\operatorname{\bf type}}
\newcommand{\ad}{\operatorname{ad}}
\newcommand{\GG}{\mathscr{G}}
\newcommand{\Fund}{\operatorname{Alt}}

\newcommand{\VV}{\mathcal{V}}
\newcommand{\Alt}{\operatorname{Alt}}

\usepackage{multirow}
\usepackage{amsmath}
\usepackage{comment}
\usepackage{tikz}
\usepackage{tikz-cd}
\usepackage{mathtools}
\usepackage{arydshln}
\usepackage[vcentermath]{youngtab}

\tikzstyle{block}=[rectangle, draw, text centered, rounded corners, minimum height=.6cm, minimum width=.6cm]
\tikzstyle{cage}=[rectangle, draw, text centered, rounded corners, dotted]

\title{Invariant Theory and wheeled PROPs}

\author{Harm Derksen and Visu Makam}

\thanks{The first author was supported by NSF grants DMS-1601229. The second author was supported by NSF grants DMS-1601229, DMS-1638352 and CCF-1412958.}

\begin{document}

\begin{abstract}
We study the category of wheeled PROPs using tools from Invariant Theory.
A typical example of a wheeled PROP is the mixed tensor algebra ${\mathcal V}=T(V)\otimes T(V^\star)$, where $T(V)$ is the tensor algebra on
 an $n$-dimensional vector space over a field of $K$ of characteristic~0. 
  First we classify all the ideals of the initial object $\ZZ$ in the category of wheeled PROPs. 
 We show that non-degenerate sub-wheeled PROPs of ${\mathcal V}$ are exactly subalgebras of the form ${\mathcal V}^G$ where $G$ is a closed, reductive subgroup of the general linear group $\GL(V)$. 
When $V$ is a finite dimensional Hilbert space,  a similar description of invariant tensors for an action of  a compact group was given by Schrijver. We also generalize the theorem of  Procesi that says that trace rings satisfying the $n$-th Cayley-Hamilton identity can be embedded in an $n \times n$ matrix ring over a commutative algebra $R$. Namely, we prove that a wheeled PROP can be embedded in $R\otimes {\mathcal V}$
for a commutative $K$-algebra $R$ if and only if it satisfies certain relations.
\end{abstract}

\maketitle

\section{Introduction}
 PROPs were introduced by Adams and MacLane (see~\cite{MacLane65}) in the context of Category Theory and formalize functors that may have several inputs and outputs.
  The abbreviation PROP stands
for {\em PROduct and Permutation category}.
Wheeled PROPs were introduced by Markl, Merkulov and Shadrin \cite{MMS09}. Besides the (tensor) product and permutations, 
wheeled PROPs also have contractions. From the viewpoint of Classical Invariant Theory, wheeled PROPs have exactly the right structure. 
Throughout this paper we will assume that $K$ is a fixed field of characteristic 0.
We will give the precise definition of a wheeled PROP in the next section, but for now we will give some important examples.

Suppose that $V$ is an $n$-dimensional $K$-vector space.
The $q$-fold tensor product is defined by 
$$
V^{\otimes q}:=\underbrace{V\otimes V\otimes \cdots\otimes V}_{q}.
$$
 By convention, $V^{\otimes 0}=K$.
Let $V^\star$ be the dual space and define ${\mathcal V}^p_q=(V^\star)^{\otimes p}\otimes V^{\otimes q}$.  There are actions of the symmetric groups $\Sigma_p$ and $\Sigma_q$ on ${\mathcal V}^p_q$. Tensor product gives a bilinear map ${\mathcal V}^{p_1}_{q_1}\times {\mathcal V}^{p_2}_{q_2}\to {\mathcal V}^{p_1+p_2}_{q_1+q_2}$. Taking partial traces give linear maps ${\mathcal V}^{p}_{q}\to {\mathcal V}^{p-1}_{q-1}$.
This combined structure makes ${\mathcal V}=\bigoplus_{p,q\geq 0}{\mathcal V}^p_q$ into a wheeled PROP. 
If $V$ is a representation of an algebraic group $G$, then the space ${\mathcal V}^G=\bigoplus_{p,q\geq 0}({\mathcal V}^p_q)^G$
of $G$-invariant tensors is also a wheeled PROP. Another example of a wheeled PROP is $R\otimes {\mathcal V}$ where
$R$ is a commutative $K$-algebra with identity.

The main goal of this paper is to develop the language of wheeled PROPs in the context of Classical Invariant Theory. 
One ingredient of the language is the use of wire diagrams (see~\cite[\S 2.11]{Landsberg12}). Such diagrams were already used by Clifford in the 19th century and also appear in the work of Feynman and Penrose. One can use a graphical calculus to do computations in representation theory. See for example
the book \cite{Cvitanovic08} of Cvitanovi\'c, who calls such diagrams bird tracks.
A tensor in ${\mathcal V}^p_q$ can be represented by a black box with $p$ inputs and $q$ outputs. 
We also will represent tensors in a more compact form using a notation  for tensors that is similar to Einstein's (see~\cite{Einstein16}).
The main results are as follows:
\begin{itemize}
\item The initial object in the category of wheeled PROPs (over $K$) is denoted by ${\mathcal Z}$. We give a complete classification of the ideals of ${\mathcal Z}$ in Section~\ref{sec4}. We also classify all prime ideals of ${\mathcal Z}$. There is a natural analogy between the category of wheeled PROPs and
the category of commutative rings with identity. The initial object in the category of commutative rings with identity is the ring of integers $\Z$.
Understanding the ideals and prime ideals of $\Z$ is essential for understanding more complicated commutative rings. Similarly, the classification
of ideals and prime ideals of ${\mathcal Z}$ is essential for understanding wheeled PROPs.
\item In Section~\ref{sec5} give an equivalence of categories between wheeled PROPs that appear as sub-wheeled PROPs of $R\otimes {\mathcal V}$ for some commutative $K$-algebra $R$ and the category of commutative $K$-algebras with a regular action of $\GL(V)$.
\item In the case where $V$ is an $n$-dimensional complex Hilbert space, Schrijver gave in \cite{Schrijver08} a correspondence between
compact subgroups of the unitary group $U\subseteq \GL(V)$ and subspaces of ${\mathcal V}$ that are closed under the permutation actions,
tensor product, contractions and the star operation (that comes from the isomorphism $V\cong \V^\star$). We formulate and prove a similar result for reductive groups instead of compact groups in Section~\ref{sec6}. This gives a correspondence between closed reductive subgroups of $\GL(V)$
and simple sub-wheeled PROPs of ${\mathcal V}$.
\item Procesi proved that a trace ring can be embedded into the matrix ring $\Mat_n(R)$ with coefficients in some commutative ring $R$ with identity if and only if the trace ring satisfies the $n$-th Cayley-Hamilton relation (see~\cite{Procesi87}). In Section~\ref{sec7}, we give a similar characterization of wheeled PROPs that
appear as sub-wheeled PROPs of $R\otimes {\mathcal V}$ for some commutative $K$-algebras. This generalizes Procesi's theorem.

\end{itemize}

In Section~\ref{sec2} we will give the definition of a wheeled PROPs and in Section~\ref{sec3} we will study their properties.

\section{Definition of wheeled PROPs}\label{sec2}
Fix an algebraically closed field $K$ of characteristic $0$.

\subsection{Mixed tensor algebra} \label{mta}
The archetypal example of a wheeled PROP is the mixed tensor algebra. We introduce the mixed tensor algebra, so as to illustrate the various features of wheeled PROPs in a very concrete fashion. 
Let $V$ be an $n$-dimensional vector space over $K$ and $V^\star$ be its dual space. The tensor algebra on $V$ is $T(V)=\bigoplus_{q=0}^\infty V^{\otimes q}$ where
$$
V^{\otimes q}:=\underbrace{V\otimes V\otimes \cdots\otimes V}_{q}
$$
is the $q$-fold tensor product.
 For $p,q\geq 0$ we define $\V^p_q :=  (V^\star)^{\otimes p} \otimes V^{\otimes q}$. In the notation ${\mathcal V}^p_q$, the upper index $p$ corresponds to the contravariant part, and the lower index $q$ corresponds to the covariant part of the tensor product.
 We consider the mixed tensor algebra 
 $$\V :=T(V^\star\oplus V)\cong T(V^\star)\otimes T(V)\cong \bigoplus\limits_{p,q \in \N} \V^p_q.$$ This is a bigraded associative algebra with multiplication $\otimes$ and unit $1 \in \V_0^0 = K$. There is another special element, the identity $\id \in \V_1^1 = \End(V)$. Let $\Sigma_n$ denote the symmetric group on $n$ letters. We have an action of $\Sigma_p \times \Sigma_q$ on $\V^p_q$ by  permuting the tensor factors. Other interesting operations on $\V$  are the contraction maps $\partial_i^j : \V^p_q \rightarrow \V^{p-1}_{q-1}$ given by 
$$
\partial_i^j(f_1 \otimes \dots \otimes f_p \otimes v_1 \otimes \dots \otimes v_q)  =  f_j(v_i)  ( f_1 \otimes \cdots f_{j-1} \otimes f_{j+1} \dots \otimes f_p \otimes v_1 \otimes \cdots v_{i-1} \otimes v_{i+1} \dots \otimes v_q).
$$

We can identify $\V^p_q$ with $\Hom(V^{\otimes p}, V^{\otimes q})$. Under this identification, $\partial_i^j$ is a partial trace.

We introduce a pictorial representation of $\V$, which will serve as a motivational tool for defining wheeled PROPs. An element in $\V^p_q$ can be thought of as a map from $V^{\otimes p}$ to $V^{\otimes q}$. We will visualize this as a black box with $p$ inputs on top and $q$ outputs at the bottom. For example, we visualize $A \in \V_1^2$ with the following picture:

\centerline{\includegraphics[height=.8in]{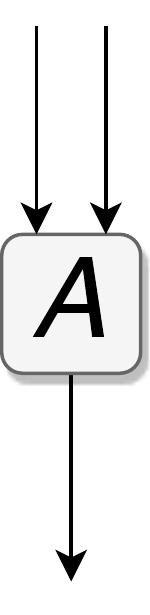}}




The contraction $\partial_i^j$ can be visualised as connecting the $j^{th}$ input and the $i^{th}$ output. 
We visualize $ \partial_1^2(A)$ by the diagram

\centerline{\includegraphics[height=1.1in]{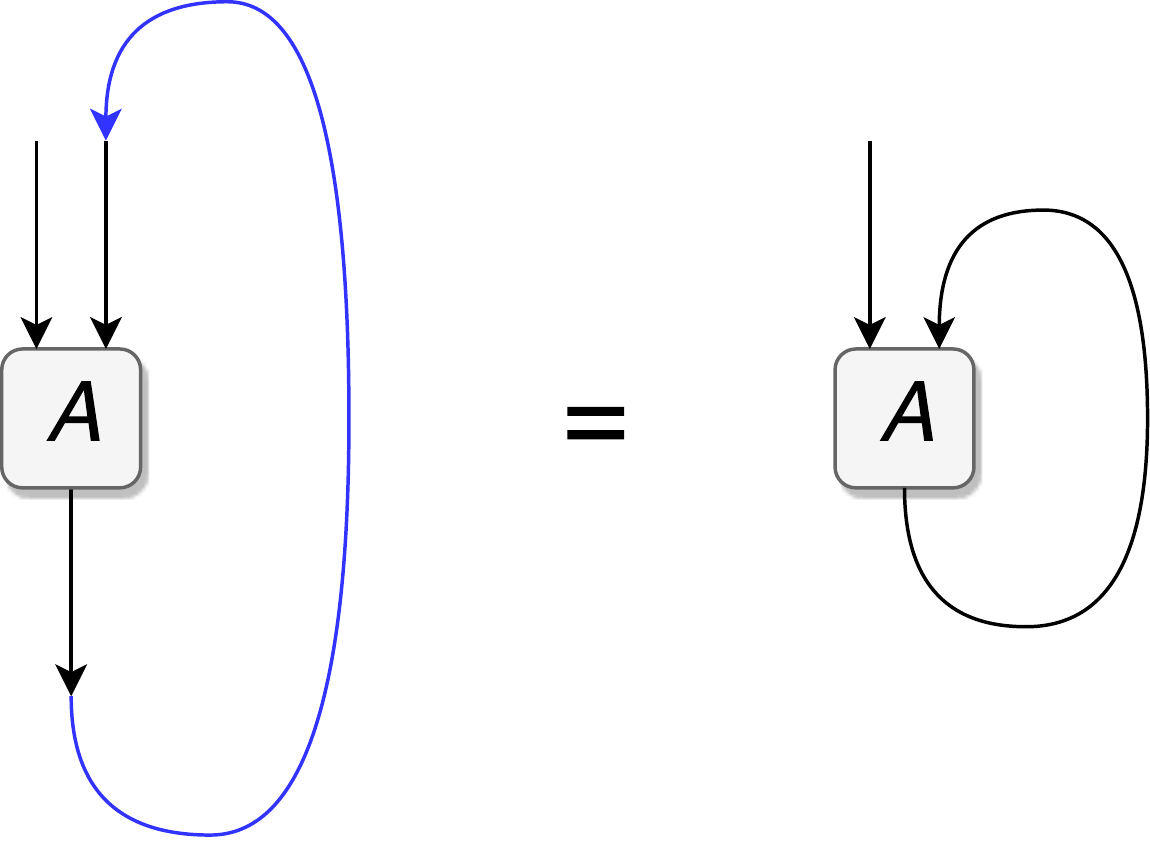}}

The tensor product corresponds to simply putting the diagrams next to each other. For example, if $A \in \V^2_1$ and $B \in \V^0_1$, then we have

\centerline{\includegraphics[height=.8in]{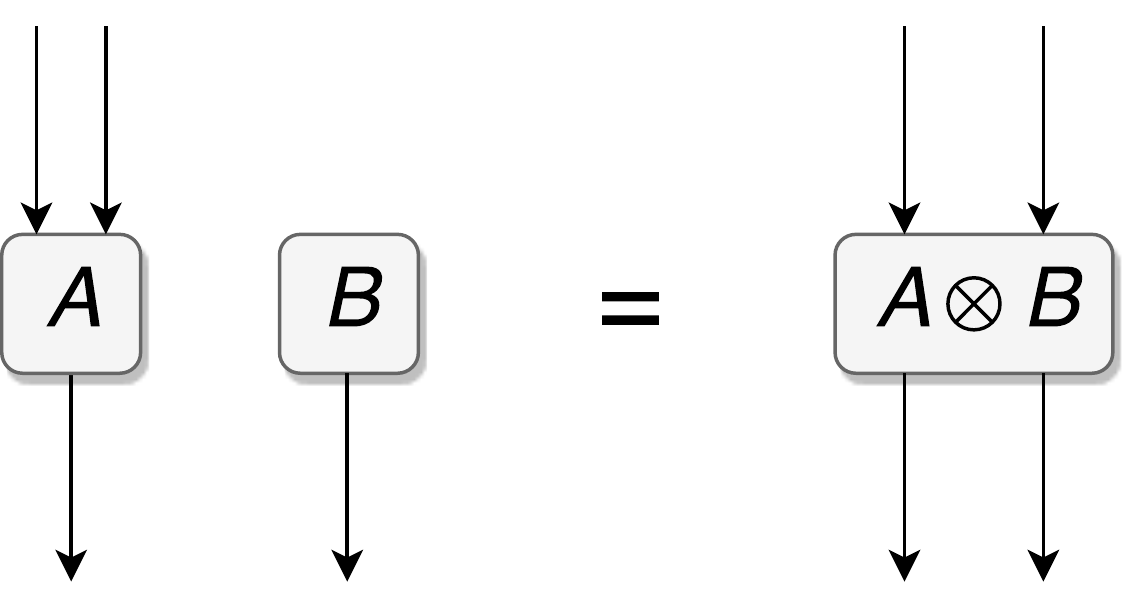}}

The space $\V_1^1 \cong \End(V)$ is a ring with unity.
For example, if $A,B\in \V_1^1$, then the product is given by $AB=\partial^1_2(A\otimes B)$:

\centerline{\includegraphics[height=1.1in]{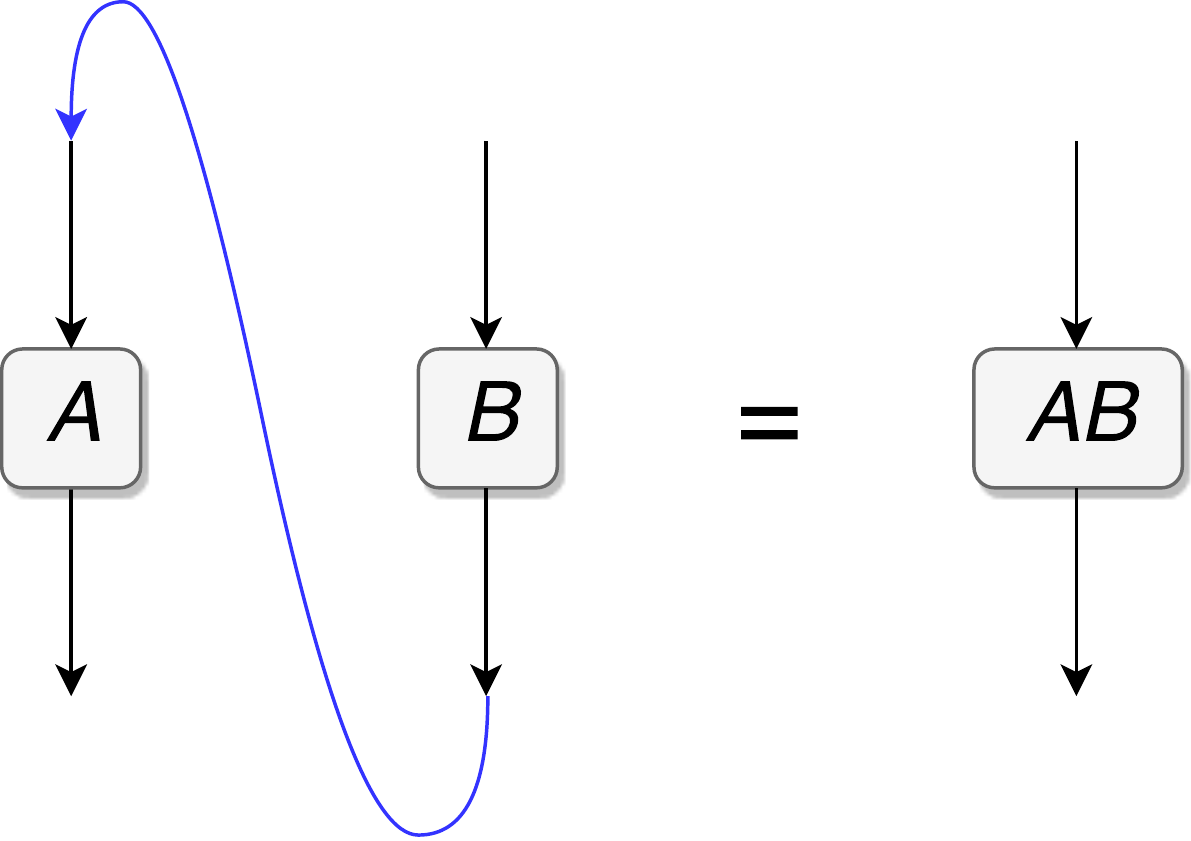}}

 The identity ${\rm id} \in \V^1_1$ will be denoted by a single directed edge without any labels:

\centerline{\includegraphics[height=.8in]{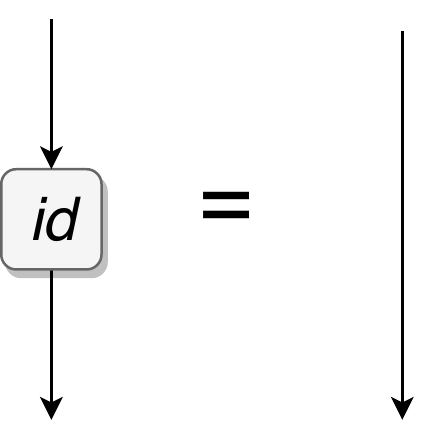}}

If we connect the output of the identity to its input, then we get $\Tr(\id)=n$, the dimension of $V$:

\centerline{\includegraphics[height=.7in]{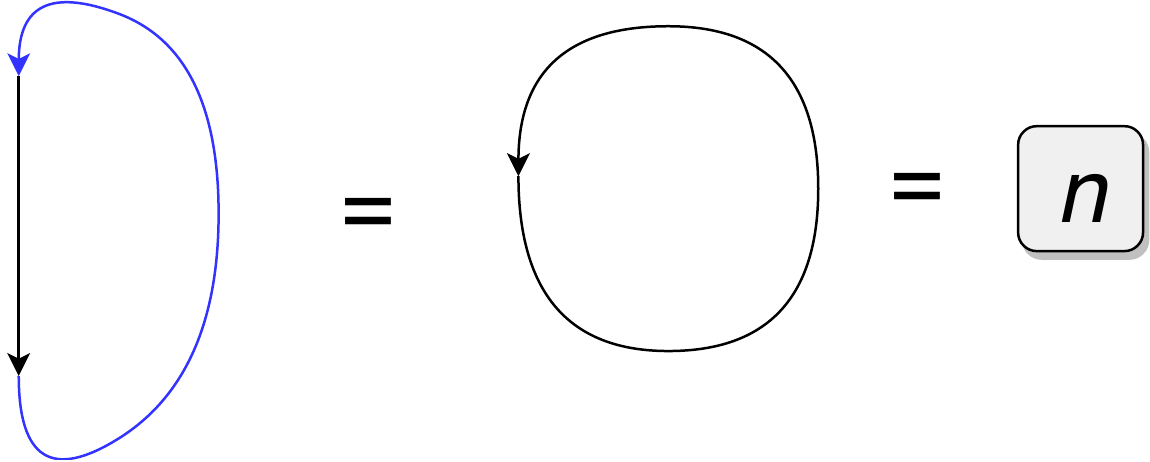}}

That the element ${\rm id} \in \V^1_1$ acts as the identity can be listed out as a set of conditions, and the following picture gives an example. 

\centerline{\includegraphics[height=1.1in]{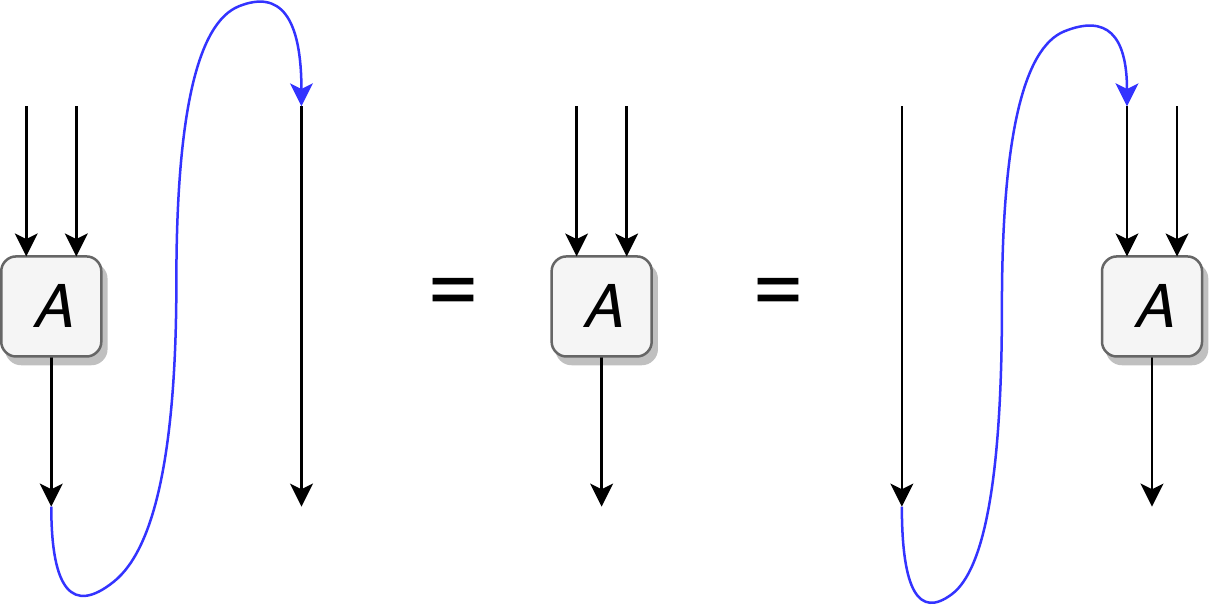}}

For $\sigma \in \Sigma_n$, we look at the map $\sigma: V^{\otimes n} \rightarrow V^{\otimes n}$ given by $v_1 \otimes \dots \otimes v_n \mapsto v_{\sigma^{-1}(1)} \otimes \dots \otimes v_{\sigma^{-1}(n)}$. For example, the permutation 3124 (in one line notation) is represented by the picture below.

\centerline{\includegraphics[height=.8in]{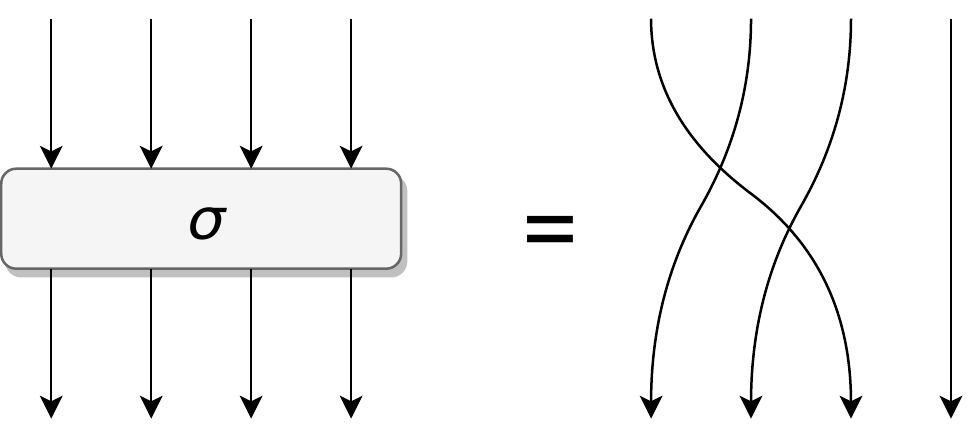}}

We denote the diagram above by $[\sigma]$.
\subsection{Pre-wheeled PROPs}
We first define pre-wheeled PROPs. Wheeled PROPs are pre-wheeled PROPs that satisfy certain axioms. Instead of giving a list of axioms,
we will define free wheeled PROPs and then define wheeled PROPs as quotients of free wheeled PROPs. Because it can be visualized, 
the notion of a free wheeled PROP is easy to grasp.
\begin{definition}
A {\em pre-wheeled PROP} is a bigraded $K$-vector space ${\mathcal R}=\bigoplus_{p,q\in \Z_{\geq 0}}{\mathcal R}^p_q$ together with
\begin{enumerate}
\item a special element $1_{\mathcal R}\in {\mathcal R}^0_0$;
\item a special element $\downarrow_{\mathcal R}\in {\mathcal R}^1_1$ called the identity;
\item a bilinear map $\otimes: {\mathcal R}^{p_1}_{q_1}\times {\mathcal R}^{p_2}_{q_2}\to {\mathcal R}^{p_1+p_2}_{q_1+q_2}$ for all $p_1,p_2,q_1,q_2\geq 0$;
\item and for all $i,j,p,q$ with $1\leq i\leq p$ and $1\leq j\leq q$ a linear map $\partial^i_j:{\mathcal R}^p_q\to {\mathcal R}^{p-1}_{q-1}$.
\end{enumerate}
\end{definition}
For a wheeled-PROP ${\mathcal R}$ there is an action of $\Sigma_p\times \Sigma_q$ on ${\mathcal R}^p_q$. Since the action
of $\Sigma_p\times \Sigma_q$ can be expressed in terms of tensor products and contractions, we will not use this action
in the definition of a pre-wheeled PROP.

\begin{definition}
A ${\mathcal R},{\mathcal S}$ are pre-wheeled PROPs, then a homomorphism $\phi:{\mathcal R}\to{\mathcal S}$ of pre-wheeled PROPs is a linear map $\phi:{\mathcal R}\to {\mathcal S}$ such that
\begin{enumerate}
\item $\phi({\mathcal R}^p_q)\subseteq {\mathcal S}^p_q$ for all $p,q\geq 0$, i.e., $\phi$ preserves the bigrading;
\item $\phi(1_{\mathcal R})=1_{\mathcal S}$;
\item $\phi(\downarrow_{\mathcal R})=\downarrow_{\mathcal S}$;
\item $\phi(A\otimes B)=\phi(A)\otimes \phi(B)$ for $A\in {\mathcal R}^{p_1}_{q_1}$, $B\in {\mathcal R}^{p_2}_{q_2}$;
\item $\phi(\partial^i_j(A))=\partial^i_j(\phi(A))$ for $A\in {\mathcal R}^p_q$.
\end{enumerate}
\end{definition}
If ${\mathcal S}$ is a pre-wheeled PROP and ${\mathcal R}^p_q\subseteq {\mathcal S}^p_q$ is a subspace for all $p,q\geq 0$, then ${\mathcal R}=\bigoplus_{p,q\geq0} {\mathcal R}^p_q$ is a sub-wheeled PROP when it contains
$1_{\mathcal R}$, $\downarrow_{\mathcal R}$ and it is closed under $\otimes$ and $\partial^i_j$ for all $i,j$. If ${\mathcal R}$ is a sub-wheeled PROP of ${\mathcal S}$  then it is easy to see that it is a pre-wheeled PROP and that  the inclusion ${\mathcal R}\hookrightarrow {\mathcal S}$ is a homomorphism of pre-wheeled PROPs.

If ${\mathcal S}$ is a pre-wheeled PROP and ${\mathscr G}\subseteq \bigcup_{p,q\geq 0} {\mathcal S}^p_q$ is a subset, then we say that ${\mathscr G}$ generates ${\mathcal S}$
if the smallest sub-wheeled PROP of ${\mathcal S}$ containing ${\mathscr G}$ is ${\mathcal S}$ itself. This is exactly the case when every element of ${\mathcal S}$
can be obtained from ${\mathscr G}\cup \{1,\downarrow\}$ by using operations $\otimes$, $\partial^i_j$ $(i,j\geq 1)$ and taking $K$-linear combinations.
It is clear that if ${\mathscr G}$ generates ${\mathcal S}$ and $\psi_1,\psi_2:{\mathcal S}\to {\mathcal R}$ are homomorphisms of pre-wheeled PROPs, then $\psi_1=\psi_2$ if and only if the restrictions of $\psi_1$ and $\psi_2$ to ${\mathscr G}$ are equal.

\begin{definition}
Suppose that ${\mathcal R}$ is a pre-wheeled PROP. An ideal ${\mathcal J}$ of ${\mathcal R}$ is a subspace ${\mathcal J}=\bigoplus_{p,q\in \Z_{\geq 0}} {\mathcal J}^p_q$
with ${\mathcal J}^p_q\subseteq {\mathcal R}^p_q$ for all $p,q$ such that
\begin{enumerate}
\item if $A\in {\mathcal R}^{p_1}_{q_1}$ and $B\in {\mathcal J}^{p_2}_{q_2}$ then $A\otimes B,B\otimes A\in {\mathcal J}^{p_1+p_2}_{q_1+q_2}$;
\item if $A\in {\mathcal J}^p_q$ then $\partial^i_j(A)\in {\mathcal J}^{p-1}_{q-1}$.
\end{enumerate}
\end{definition}
The following lemma is left to the reader.
\begin{lemma}\ 
\begin{enumerate}
\item
If $\phi:{\mathcal R}\to {\mathcal S}$ is a homomorphism of pre-wheeled PROPs, then $\ker(\phi)$ is an ideal;
\item if ${\mathcal J}$ is an ideal of the pre-wheeled PROP ${\mathcal R}$ then the quotient ${\mathcal R}/{\mathcal J}=\bigoplus_{p,q\in \Z_{\geq 0}} {\mathcal R}^p_q/{\mathcal J}^p_q$ has the structure of a pre-wheeled PROP where
\begin{enumerate}
\item $1_{{\mathcal R}/{\mathcal J}}=1_{\mathcal R}+{\mathcal J}^0_0\in {\mathcal R}^0_0/{\mathcal J}^0_0$;
\item $\downarrow_{{\mathcal R}/{\mathcal J}}=\downarrow_{\mathcal R}+{\mathcal J}^1_1\in {\mathcal R}^1_1/{\mathcal J}^1_1$;
\item $(A+{\mathcal J}^{p_1}_{q_1}) \otimes (B+{\mathcal J}^{p_2}_{q_2})=A\otimes B+{\mathcal J}^{p_1+p_2}_{q_1+q_2}$ when $A\in {\mathcal R}^{p_1}_{q_1}$ and $B\in {\mathcal R}^{p_2}_{q_2}$;
\item $\partial^i_j(A+{\mathcal J}^p_q)=\partial^i_j(A)+{\mathcal J}^{p-1}_{q-1}$ when $A\in {\mathcal R}^p_q$.
\end{enumerate}
\end{enumerate}
\end{lemma}
It is easy to see that ${\mathcal V}$ defined in Section~\ref{mta} has the structure of a pre-wheeled PROP. 
\subsection{Free wheeled PROPs}
To define wheeled PROPs, we will first define {\em free} wheeled PROPs. Arbitrary wheeled PROPs  will then  be defined as pre-wheeled PROPs that are quotients of wheeled PROPs.
To construct a free wheeled PROP, we start with a  set ${\mathscr G}$ of generators, and a function $\type:{\mathscr G}\to \Z_{\geq 0}^2$.  
We also fix a countable infinite set ${\mathcal X}$ of variables. A generator $A\in {\mathscr G}$ with $\type(A)={p\choose q}$ will be graphically represented as a black box labeled $A$ with $p$ inputs (next to each other) and $q$ outputs. 
(The type is called {\em biarity} in \cite{MMS09}.)

\centerline{\includegraphics[height=.8in]{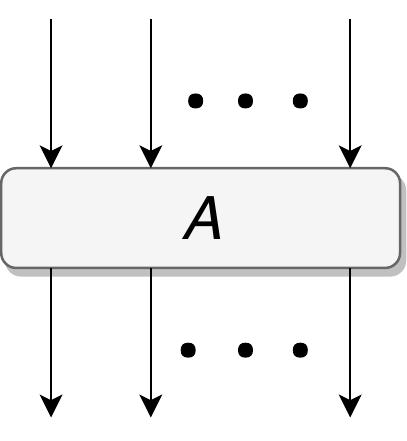}}


\begin{definition}
An {\em atom} is an expression of one of the following forms:
\begin{enumerate}
\item $A^{x_1,x_2,\dots,x_p}_{y_1,y_2,\dots,y_q}$ where $A$ is a generator of type ${p\choose q}$, $x_1,x_2,\dots,x_p$ are distinct variables, and $y_1,y_2,\dots,y_q$ are distinct variables;
\item $\downarrow^{x}_y$ where $x$ and $y$ are variables.
\end{enumerate}
\end{definition}
Variables that appear as an upper index will be referred to as {\em input variables} and variables that appear as lower index are called {\em output variables}.
An atom $A^{x_1,x_2,\dots,x_p}_{y_1,y_2,\dots,y_q}$ of the first kind is graphically represented as the generator $A$ where the inputs are labeled $x_1,x_2,\dots,x_p$ clockwise
 and the outputs are labeled $y_1,y_2,\dots,y_q$ counterclockwise.
 
 \centerline{\includegraphics[height=.8in]{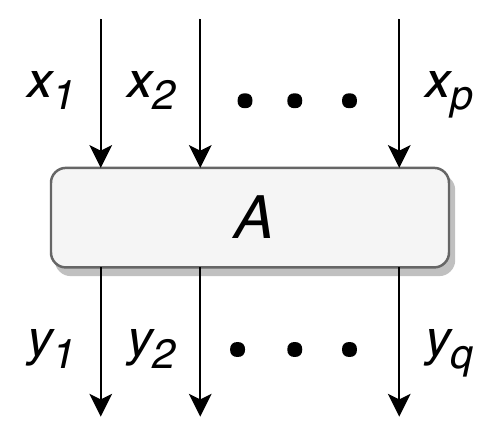}}
 
An atom $\downarrow^x_y$ of the second kind will graphically be represented by an  arrow where the tail is labeled by $x$ and the head is labeled by $y$:
\begin{center}
\begin{tikzpicture}
\draw[->] (0,1) to (0,-1);
\node at (-.2,.7) (a) {$\scriptstyle x$};
\node at (-.2,.-.7) (b) {$\scriptstyle y$};
\end{tikzpicture}
\end{center}
\begin{example}\label{ex:ex1}

Suppose that ${\mathscr G}=\{A,B\}$ and $A$ and $B$ are of type ${2\choose 1}$ and ${0\choose 1}$ respectively, and ${\mathcal X}=\{w,x,y,z,\dots\}$. Then examples of atoms are
$$
A^{x,y}_z, A^{z,y}_z, B_x, \downarrow^x_y, \downarrow^z_z.
$$
But $A^{x,x}_y$ is not an atom because the variable $x$ appears as an input twice.
\end{example}

\begin{definition}
A {\em molecule} is an unordered finite sequence $M$ (or set) of atoms where
 every variable appears at most once as an input variable of an atom and  at most once as an output variable of an atom in $M$.
 An {\em input variable} of the molecule $M$  is a variable that appears as an input of an atom, but not as an output of an atom in $M$.
An {\em output variable} of $M$ is a variable that appears as an output of an atom, but not as an input of an atom in $M$.
A {\em bound variable} is a variable that appears as an input variable of an atom, and as an output variable of an atom in the molecule. 
A {\em free variable} is a variable that is an input or an output variable. If $x_1,x_2,\dots,x_p$ are the input variables of $M$ ordered from small to large,
and $y_1,y_2,\dots,y_q$ are the output variables of $M$ ordered from small to large then we denote the molecule by $M^{x_1,x_2,\dots,x_p}_{y_1,y_2,\dots,y_q}$.
\end{definition}
\begin{example}\label{ex:ex2}
Let ${\mathscr G}=\{A,B\}$ and ${\mathcal X}$ as in Example~\ref{ex:ex1}. The expression $A^{x,z}_y \downarrow^y_wB_x$ is a
molecule where $\{x,y\}$ are the bound variables, $z$ is an input variable, $w$ is an output variable and $\{z,w\}$ are the free variables.
The expression $\downarrow^s_s A^{z,x}_y A^{y,w}_w \downarrow^t_z$ is a molecule with bound variables $\{w,y,z\}$, input variables $\{t,x\}$ and no output variables. 
\end{example}
Graphically, we represent a molecule by first drawing all the atoms in the molecule. Whenever some variable $x$ appears as an output and input label, then
we connect that output and input. It may be unavoidable that some connections intersect.  

\begin{example}
The diagram of the molecule $\downarrow^s_s A^{z,x}_y A^{y,w}_w \downarrow^t_z$ from Example~\ref{ex:ex2} is:

\centerline{\includegraphics[height=1.4in]{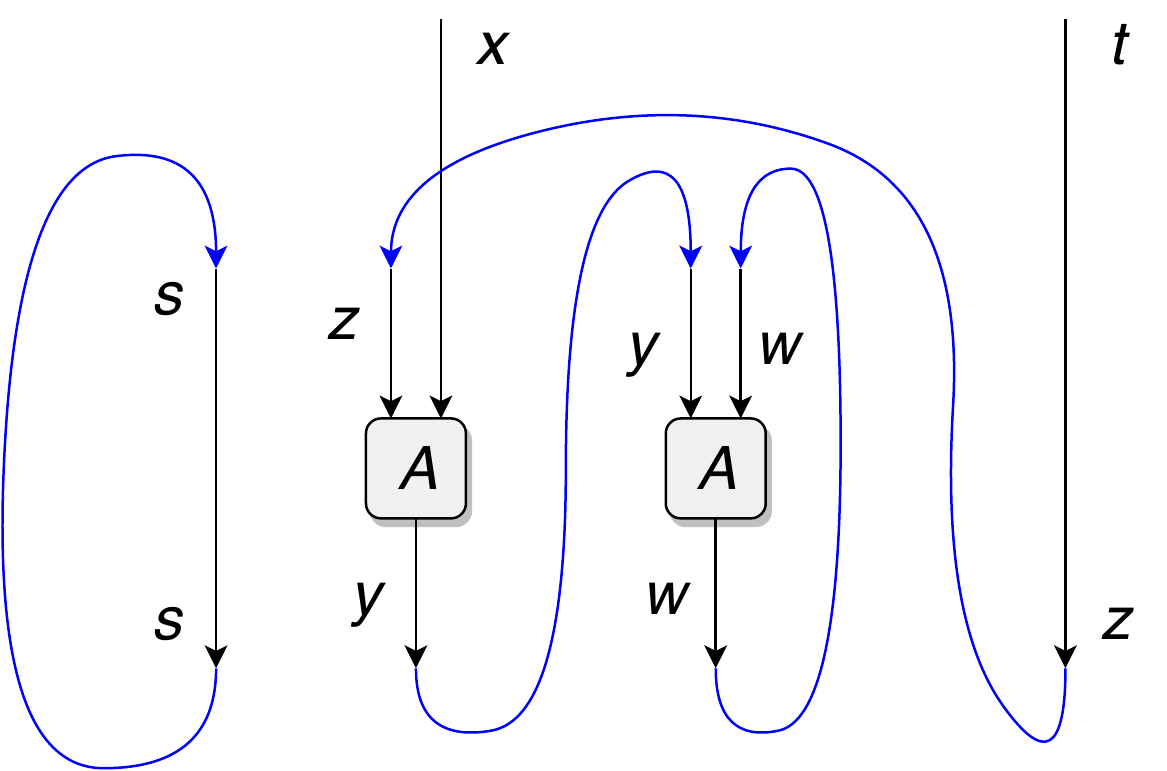}}

%
%
%
%
%
\end{example}
\begin{definition}
We define an equivalence relation $\sim$ on molecules. The equivalence relation is generated by the following rules:
$M^{x_1,\dots,x_p}_{y_1,\dots,y_q}$ is equivalent to $N^{x_1,\dots,x_p}_{y_1,\dots,y_q}$ if
$N$ is obtained from $M$ by
\begin{enumerate}
\item replacing $\downarrow^x_y\downarrow^y_z$ by $\downarrow^x_z$;
\item replacing $\downarrow^{y_i}_zA^{x_1,x_2,\dots,x_p}_{y_1,y_2,\dots,y_q}$ by $A^{x_1,x_2,\dots,x_p}_{y_1,\dots,y_{i-1},z,y_{i+1},\dots,y_q}$ for $A\in {\mathscr G}$ of type ${p\choose q}$ and all $i$;
\item replacing $\downarrow^{z}_{x_i}A^{x_1,x_2,\dots,x_p}_{y_1,y_2,\dots,y_q}$ by $A^{x_1,\dots,x_{i-1},z,x_{i+1},\dots x_p}_{y_1,y_2,\dots,y_q}$  for $A\in {\mathscr G}$ of type ${p\choose q}$ and all $i$.
\end{enumerate}
Moreover, the molecules
$M^{x_1,\dots,x_p}_{y_1,\dots,y_q}$ and $N^{s_1,\dots,s_p}_{t_1,\dots,t_q}$ are equivalent when
\begin{itemize}
\item[(4)] there exists a bijection $\varphi:{\mathcal X}\to {\mathcal X}$ such that $\varphi(z)=z$ for every free variable $z$ of $M$ 
and $N$ is obtained from $M$ by replacing $z$ by $\varphi(z)$ for every bound variable of $M$.
\end{itemize}
 The equivalence clase of $M$ is denoted by $[M]$. If $M$ and $N$ are molecules that do not have common input variables, or common output variables,
 then we can define the product $[M][N]$ as follows. We can choose molecules $M'$ and $N'$ such that $M\sim M'$, $N\sim N'$
 such that the bound variables of $M'$ do not appear in $N'$ and the bound variables of $N'$ do not appear in $M'$. Then we define
 $[M][N]=[M'N']$. It is easy to verify that this product is well-defined and commutative. If $M^{(1)},M^{(2)},\dots,M^{(r)}$ are molecules
 such that every variable appears at most once as an input variable, and every variable appears at most once as an output variable. Then
 the product $[M^{(1)}][M^{(2)}]\cdots [M^{(r)}]$
 is well-defined and independent of the order in which we multiply the equivalence classes of molecules. Indeed, we can choose molecules
 $N^{(1)},\dots,N^{(r)}$ such that for all $i\neq j$ no bound variable of $N^{(i)}$ appears as a variable in $N^{(j)}$. Then
 we have $[M^{(1)}][M^{(2)}]\cdots [M^{(r)}]=[N^{(1)}N^{(2)}\cdots N^{(r)}]$.
\end{definition}

Suppose that $M$ is a molecule and $[M]$ is its equivalence class. By rules (1)--(3) we can find an equivalent molecule $M'$ such that for every atom of the form $\downarrow^x_y$ with $x$ and $y$ distinct, $x$ is an input variable and $y$ is an output variable of $M'$. Such a molecule $M'$ we will call {\em reduced}. 
To draw the diagram of $[M]$ we  draw the diagram of $M'$ where we omit all the labels that are bound variables, because the equivalence class $[M]$ does not depend on the labels of the bound variables in $M'$.

\begin{example}
The diagram  of the equivalence class $[\downarrow^s_s A^{z,x}_y A^{y,w}_w\downarrow^t_z]$ of the molecule from  Example~\ref{ex:ex2} is 

\centerline{\includegraphics[height=.8in]{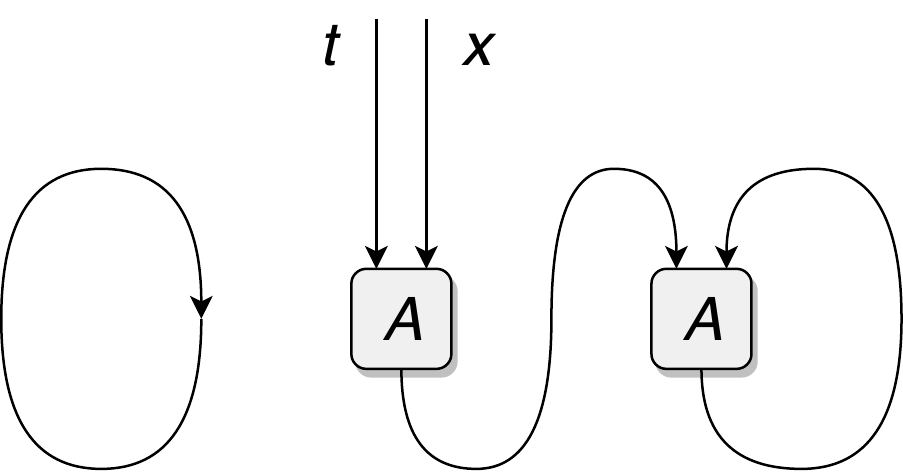}}

\begin{definition}
A {\em momomial} is an equivalence class of a molecule together with a total ordering on the input variables, and a total ordering on the output variables. If the molecule $M$ has input variables $x_1,\dots,x_p$ and output variables $y_1,\dots,y_q$ then the monomial with equivalence class $[M]$ and orderings $x_1<x_2<\cdots<x_p$ and $y_1<y_2<\cdots<y_q$ will be denoted by
$[M]^{x_1,x_2,\dots,x_p}_{y_1,y_2,\dots,y_q}$.
\end{definition}
If $A\in {\mathscr G}$ is a variable of type ${p\choose q}$ then we can view $A$ as a monomial by identifying it with $[A^{x_1,\dots,x_p}_{y_1,\dots,y_q}]^{x_1,\dots,x_p}_{y_1,\dots,y_q}$.

\end{example}
\begin{definition}
For a set ${\mathscr G}$ of generators and a function $\type:{\mathscr G}\to \Z_{\geq 0}^2$ we define the free wheeled PROP generated by ${\mathscr G}$
as the bigraded vector space ${\mathcal Z}\langle {\mathscr G}\rangle=\bigoplus_{p,q\in \Z_{\geq 0}}{ \mathcal Z}\langle {\mathscr G}\rangle^p_q$ where ${ \mathcal Z}\langle {\mathscr G}\rangle^p_q$ is the $K$-vector space with a basis consisting of all monomials with $p$ inputs and $q$ outputs. To give ${\mathcal Z}\langle {\mathscr G}\rangle $ the structure of a pre-wheeled PROP, we define:
\begin{enumerate}
\item $1=[\,]=[\emptyset]$ (equivalence class of the empty module);
\item $\downarrow=[\downarrow^x_y]^x_y$ (where $x,y$ are distinct variables);
\item if $M^{x_1,\dots,x_p}_{y_1,\dots,y_q}$ and $N^{z_1,\dots,z_r}_{w_1,\dots,w_s}$ have disjoint variables, then we define
$$[M]^{x_1,\dots,x_p}_{y_1,\dots,y_q}\otimes [N]^{z_1,\dots,z_r}_{w_1,\dots,w_s}=[M][N]^{x_1,\dots,x_p,z_1,\dots,z_r}_{y_1,\dots,y_q,w_1,\dots,w_s}.
$$
In terms of diagrams, we get the diagram of $[M]^{x_1,\dots,x_p}_{y_1,\dots,y_q}\otimes [N]^{z_1,\dots,z_r}_{w_1,\dots,w_s}$ by drawing the diagram of $[N]^{z_1,\dots,z_r}_{w_1,\dots,w_s}$
to the right of the diagram of $[M]^{x_1,\dots,x_p}_{y_1,\dots,y_q}$.

\item we define
$$
\partial^i_j[M]^{x_1,\dots,x_p}_{y_1,\dots,y_q}=[M']^{x_1,\dots,x_{i-1},x_{i+1},\dots,x_p}_{y_1,\dots,y_{j-1},y_{j+1},\dots,y_{q}},
$$
where $M'$ is obtained from $M$ by replacing $y_j$ by $x_i$.
So the diagram of $\partial^i_j[M]^{x_1,\dots,x_p}_{y_1,\dots,y_q}$ is obtained by connecting the $j$-th output of $[M]^{x_1,\dots,x_p}_{y_1,\dots,y_q}$ to the $i$-th input.
\end{enumerate}
\end{definition}
\begin{definition}
A {\em wheeled PROP} is a pre-wheeled PROP ${\mathcal A}$ such that there exists a surjective homomorphism of pre-wheeled PROPs ${\mathcal Z}\langle {\mathscr G}\rangle\to {\mathcal A}$
for some set of generators ${\mathscr G}$. If ${\mathcal R}$ and ${\mathcal S}$ are wheeled PROPs then a homomorphism $\psi:{\mathcal R}\to {\mathcal S}$ is just a homomorphism of pre-wheeled PROPs.
\end{definition}
For a free wheeled PROP ${\mathcal Z}\langle {\mathscr G}\rangle\to {\mathcal A}$ there is an
action of $\Sigma_p\times \Sigma_q$ on ${\mathcal Z}\langle {\mathscr G}\rangle^p_q$ by permuting the inputs and outputs.
the action of an element in $\Sigma_p$ or $\Sigma_q$ can be constructed by taking the tensor product with copies of $\downarrow$ and doing certain contractions. For example,  $\partial^i_{q+1}(A\otimes \downarrow_{\mathcal R})$ is equivalent to applying the cyclic permutation 
 $(i+1\ i+2\ \cdots\ p\ i)$ to the inputs. Such cycles generate $\Sigma_p$. A similar thing can be done for the outputs.
 In particular, ideals in the free wheeled PROP are stable under the action of $\Sigma_p\times\Sigma_q$
and $\Sigma_p\times \Sigma_q$ also acts on the quotient. This means that for every wheeled PROP ${\mathcal R}$ we have
an action of $\Sigma_p\times\Sigma_q$ on ${\mathcal R}^p_q$.
 By expressing the action of $\Sigma_p\times \Sigma_q$ in terms of tensor product and contractions, 
 we see that  for a homomorphism of wheeled PROPs $\phi:{\mathcal R}\to {\mathcal S}$ the map $\phi:{\mathcal R}^p_q\to {\mathcal S}^p_q$
is equivariant with respect to the group $\Sigma_p\times \Sigma_q$. Also, if ${\mathcal S}$ is a sub-wheeled PROP or an ideal of ${\mathcal R}$,
 then ${\mathcal S}^p_q$ is stable under the action of $\Sigma_p\times \Sigma_q$ for all $p,q$.

\section{Properties of free wheeled PROPs}\label{sec3}
\subsection{The universal property}
\begin{proposition}[Universal property of free wheeled PROPs]
Suppose that ${\mathscr G}$ is a generator set, $\type:{\mathscr G}\to \Z_{\geq 0}^2$ is a function,  
${\mathcal R}$ is a wheeled PROP and $\psi:{\mathscr G}\to {\mathcal R}$ is a function such that  $\psi(A)\in {\mathcal R}^p_q$ for all $A\in {\mathscr G}$ with $\type(A)={p\choose q}$. 
Then $\psi$ extends uniquely to a homomorphism of wheeled PROPs $\psi:{\mathcal Z}\langle {\mathscr G}\rangle\to {\mathcal R}$.
\end{proposition}
\begin{proof}
The extension, if it exists, is unique because ${\mathcal Z}\langle {\mathscr G}\rangle$ is generated by ${\mathscr G}$. We only need to show the existence of the extension.
For this, we may assume without loss of generality that ${\mathcal R}={\mathcal Z}\langle {\mathscr H}\rangle$ is a free wheeled-PROP. 

For simplicity, let us first assume that $\psi(A)$ is a monomial for every $A\in {\mathscr G}$. If $[M]^{x_1,\dots,x_p}_{y_1,\dots,y_p}\in {\mathcal Z}\langle {\mathscr G}\rangle^p_q$ is a monomial,
then the diagram of $\psi([M]^{x_1,\dots,x_p}_{y_1,\dots,y_p})$ is
obtained from the diagram of $[M]^{x_1,\dots,x_p}_{y_1,\dots,y_p}$ by replacing every atom $A^{s_1,\dots,s_q}_{t_1,\dots,t_r}$ appearing in the diagram
by the diagram of the monomial $\psi(A)$.

Suppose that for example,  ${\mathscr G}=\{A\}$ with $A$ of type ${2\choose 2}$, ${\mathscr H}=\{B\}$ with $B$ of type ${1\choose 1}$ and $\psi(A)=[\downarrow^x_w B^y_z]^{x,y}_{z,w}$.
In diagrams this means that

\centerline{\includegraphics[height=.8in]{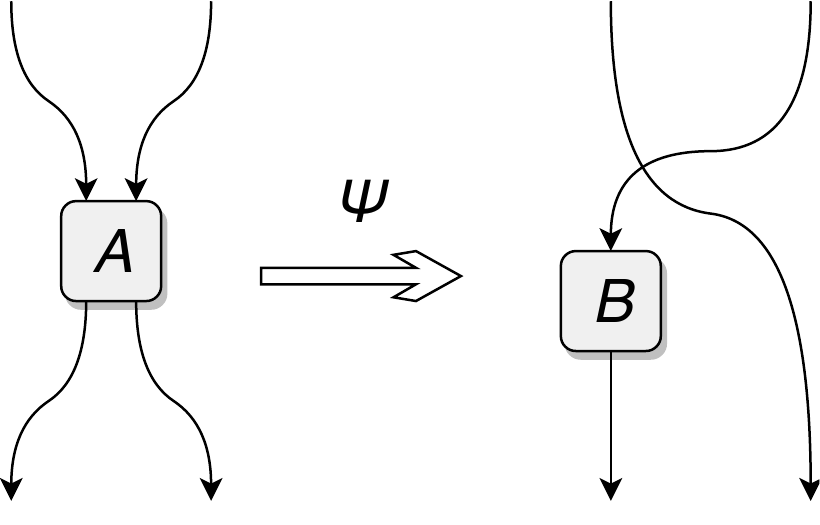}}

So we get  $\Psi([A^{x,y}_{x,z} A^{v,w}_{y,v}]^w_z)=[\downarrow^x_zB^y_x \downarrow^v_vB^w_y]^w_z=[B^y_xB^w_y\downarrow^v_v]^w_z$ and in diagrams:

\centerline{\includegraphics[height=1.4in]{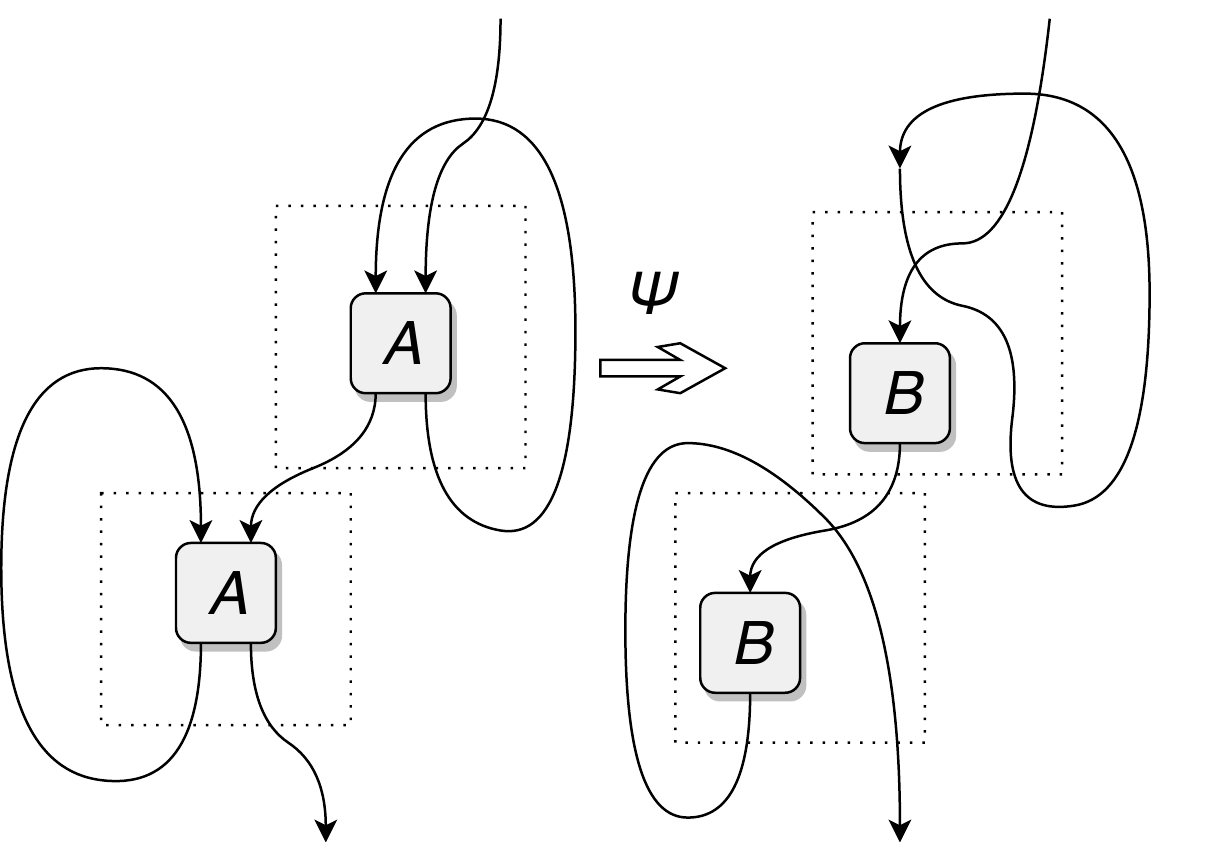}}

Since the monomials form a basis of ${\mathcal Z}\langle {\mathscr G}\rangle$,  $\psi$ extends to a $K$-linear map $\psi:{\mathcal Z}\langle{\mathscr G}\rangle\to {\mathcal Z}\langle {\mathscr H}\rangle$. It is easy to verify that $\psi$ is a homomorphism of wheeled PROPs.

If $\psi(A)$ is not a monomial for every generator $A\in {\mathscr G}$ then we define $\psi([M]^{x_1,\dots,x_p}_{y_1,\dots,y_p})$ by viewing it as a multi-linear expression
in the atoms that appear in $[M]^{x_1,\dots,x_p}_{y_1,\dots,y_p}$. For example, suppose that ${\mathcal G}=\{A\}$ with $A$ of type ${2\choose 2}$,
${\mathcal H}=\emptyset$ and $\psi(A)=2[\downarrow^x_z\downarrow^y_w]^{x,y}_{z,w}-[\downarrow^x_w\downarrow^y_z]^{x,y}_{z,w}$. So we have

\centerline{\includegraphics[height=.8in]{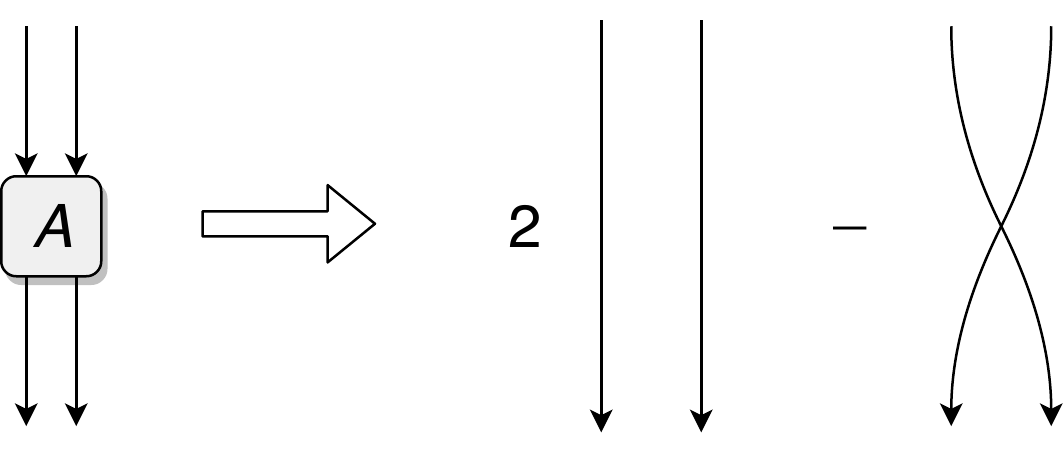}}

Using multilinearity, we get 
\begin{multline*}
$$\psi([ A^{x,y}_{z,w} A^{w,t}_{u,v}]^{x,y,t}_{z,u,v} )=[(2\downarrow^x_z\downarrow^y_w-\downarrow^x_w\downarrow^y_z)(2\downarrow^w_u\downarrow^t_v-\downarrow^w_v\downarrow^t_u)]^{x,y,t}_{z,u,v}=\\=
4[\downarrow^x_z\downarrow^y_w\downarrow^w_u\downarrow^t_v]^{x,y,t}_{z,u,v}-2[\downarrow^x_z\downarrow^y_w\downarrow^w_v\downarrow^t_u]^{x,y,t}_{z,u,v}-2[\downarrow^x_w\downarrow^y_z\downarrow^w_u\downarrow^t_v]^{x,y,t}_{z,u,v}+[\downarrow^x_w\downarrow^y_z\downarrow^w_v\downarrow^t_u]^{x,y,t}_{z,u,v}=\\
4[\downarrow^x_z\downarrow^y_u\downarrow^t_v]^{x,y,t}_{z,u,v}-2[\downarrow^x_z\downarrow^y_v\downarrow^t_u]^{x,y,t}_{z,u,v}-2[\downarrow^x_u\downarrow^y_z\downarrow^t_v]^{x,y,t}_{z,u,v}+[\downarrow^x_v\downarrow^y_z\downarrow^t_u]^{x,y,t}_{z,u,v}.
\end{multline*}
Using the linearity of $\partial^i_j$ and the bi-linearity of $\otimes$ it is easy to verify that  extending $\psi$ to a linear map gives a homomorphism of wheeled PROPs.

\end{proof}
\subsection{The initial object $\ZZ$}
From the universal property it is clear that the free wheeled PROP  $\ZZ=\ZZ\langle \emptyset \rangle$ is the initial object in the category of
wheeled PROPs (and also in the category of  pre-wheeled PROPs). This means that for every wheeled PROP ${\mathcal P}$ there exists a unique homorphism of wheeled PROPs $\psi:\ZZ\to{\mathcal P}$.

 If there are no generators, the only atoms are of the form $\downarrow^x_y$.
So all the molecules, up to equivalence, are of the form
$$
\downarrow^{x_1}_{y_1}\downarrow^{x_2}_{y_2}\cdots \downarrow^{x_p}_{y_p} \downarrow^{z_1}_{z_1}\downarrow^{z_2}_{z_2}\cdots \downarrow^{z_r}_{z_r},
$$
where $x_1,x_2,\dots,x_p,y_1,y_2,\dots,y_p,z_1,z_2,\dots,z_r$ are distinct variables. Now $[\downarrow^z_z]$ corresponds to the loop diagram $\circlearrowright$.
We set $t=[\downarrow^z_z]\in \ZZ^0_0$,
and for $\sigma\in \Sigma_p$ we define 
$$
[\sigma]=[\downarrow^{x_1}_{y_{\sigma(1)}}\downarrow^{x_2}_{y_{\sigma(2)}}\cdots \downarrow^{x_p}_{y_{\sigma(p)}}]^{x_1,x_2,\dots,x_p}_{y_1,y_2,\dots,y_p}.
$$
From the description of the molecules follows 
\begin{corollary}
If $p\neq q$ then ${\mathcal Z}^p_q=0$, and  ${\mathcal Z}^p_p$ is the free $K[t]$-module with basis $[\sigma]$, $\sigma\in \Sigma_p$.
In particular ${\mathcal Z}^0_0=K[t]$. 
\end{corollary}
\begin{remark}
There is an analogy between the category of wheeled PROPs and the category of commutative rings with identity.
The initial object in the category of commutative rings is the ring of integers ${\mathbb Z}$. For any commutative ring $R$ with identity
there exists a unique ring homomorphism $\psi:{\mathbb Z}\to R$. The kernel is an ideal of ${\mathbb Z}$ of the form $(p)$
where $p$ is a nonnegative integer. If $R$ is an integral domain then $p=0$ or $p$ is prime and $p$ is  the characteristic of $R$.
From this it is clear that the understanding of the ideals of the ring of integers is essential for studying the category of commutative rings with identity in general. Using the analogy, it is clear that studying the ideals and  prime ideals (which we will define later) in the wheeled PROP ${\mathcal Z}$ is fundamental. 
Using the representation theory of symmetric groups, we classify all the ideals of $\ZZ$. 
\end{remark}

\begin{theorem} [Theorem~\ref{ideals} rephrased] \label{informal}
There is a $1-1$ correspondence between ideals in $\ZZ$ and tuples of the form $(f, S)$ where $f \in K[t]$ is a monic polynomial, and $S$ is a finite subset of $\Z_{>0}^2$.
\end{theorem}

We will refer to the ideal corresponding to $(f,S)$ by ${\mathcal I}(f,S)$

\subsection{Representations of wheeled PROPs and Lie algebras}
\begin{definition}
A {\em representation} of a wheeled PROP ${\mathcal R}$ is a homomorphism of wheeled PROPs $\psi:{\mathcal R}\to {\mathcal V}$ where
${\mathcal V}$ is the mixed tensor algebra on some finite dimensional vector space $V$.
\end{definition}
When $\psi:{\mathcal R}\to {\mathcal V}$ is a representation then $V$ is called an algebra over ${\mathcal R}$. One can define various wheeled PROPs (or just  PROPs)
that incorporate the axioms of certain types of algebras, for example associative algebras, Lie algebras, Jordan algebras and so forth. As an example, we will see next
that one can used wheeled PROPs to axiomatize {\em semi-simple} Lie algebras.

The structure of Lie algebras can be captured using PROPs or operads. As an illustration of {\em wheeled} PROPs we will also show how
the structure of finite dimensional {\it semisimple} Lie algebras can be captured.
Suppose $V$ is a finite dimensional $K$-vector space, and $[\cdot,\cdot]:V\times V\to V$ is a Lie bracket. We can identify $[\cdot,\cdot]$ with a tensor $L\in V^\star\otimes V^\star\otimes V$. 
The well-known axioms of a Lie algebra are
\begin{enumerate}
\item $[a,b]=-[b,a]$;
\item $[a,[b,c]]+[b,[c,a]]+[c,[a,b]]=0$.
\end{enumerate}
The first axiom translates to $L^{a,b}_c+L^{b,a}_c=0$:

\centerline{\includegraphics[height=.8in]{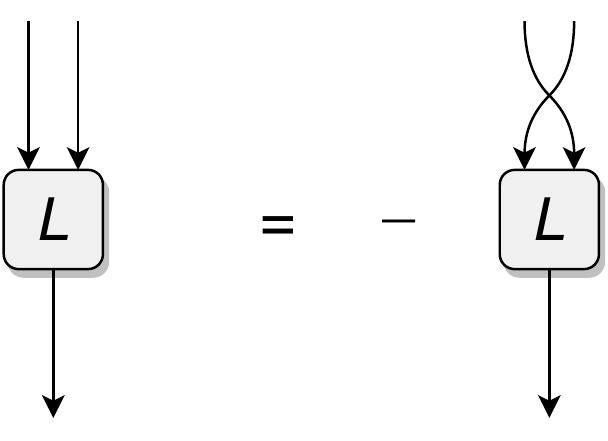}}

 and the second axiom translates to $L^{a,d}_eL^{b,c}_d+L^{b,d}_eL^{c,a}_d+L^{c,d}_eL^{a,b}_d=0$ which is, using (1), equivalent to
 $L^{a,d}_eL^{b,c}_d=L^{a,c}_dL^{b,d}_e+L^{a,b}_dL^{c,d}_e$:

\centerline{\includegraphics[height=1.3in]{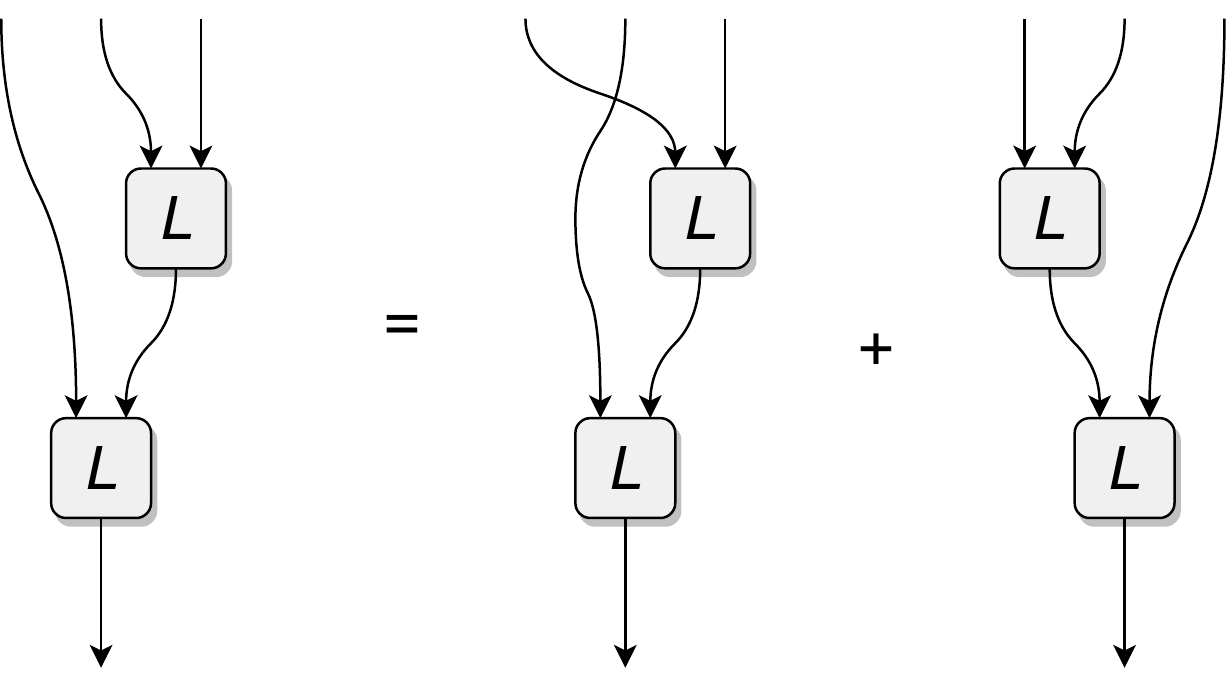}}

The Killing form $\kappa:(a,b)\mapsto \Tr(\ad(a)\ad(b))$ can be viewed as a tensor in $V^\star\otimes V^\star$ and is equal to $\kappa^{a,b}=L^{a,c}_dL^{b,d}_c$.

\centerline{\includegraphics[height=1.2in]{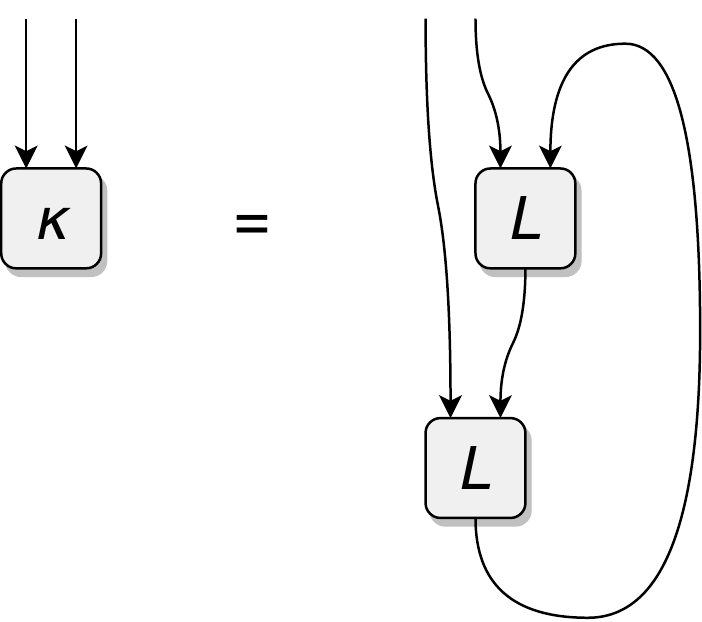}}

The Lie algebra is semi-simple if and only if the Killing form is nondegenerate. This is the case when $\kappa$ as a linear map from $V$ to $V^\star$ has an inverse $C:V^\star\to V$
which we can view as a tensor in $V\otimes V$. The tensor $C\in S^2(V)\subseteq V\otimes V$ can be thought of as the Casimir element.
  Saying that $C$ is the inverse of $\kappa$ translates to $L^{a,c}_dL^{b,d}_cC_{b,e}=\kappa^{a,b}C_{b,e}=\downarrow^a_e$. 
Let ${\mathcal R}={\mathcal Z}\langle L,C\rangle/{\mathcal J}$ be the wheeled PROP generated by $L$ of type ${2\choose 1}$ and $C$ of type ${0\choose 2}$
modulo the ideal ${\mathcal J}$ generated by $L^{a,b}_c+L^{b,a}_c$, $L^{a,d}_eL^{b,c}_d+L^{b,d}_eL^{c,a}_d+L^{c,d}_eL^{a,b}_d$ and $L^{a,c}_dL^{b,d}_cC_{b,e}-\downarrow^a_e$.
Then semi-simple Lie algebra structures on $V$ correspond to homomorphisms of wheeled PROPs from ${\mathcal R}$ to ${\mathcal V}$ (mixed tensor algebra over $V$).

Using the Killing form $\kappa$ and its inverse $C$, we can identify $V$ with $V^\star$ and view the Lie bracket $L\in \VV^2_1$ in $\VV^3_0$. Using the relations one can verify that, as a tensor in $\VV^3_0$, $L$ is alternating.

\section{The classification of ideals in $\ZZ$}\label{sec4}

\subsection{Representations of symmetric groups}
We recall some useful results on the representation theory of symmetric groups. Standard references for this subject are \cite{FHbook, Sagan, Macdonald}. It is well known that the irreducible representations of $\Sigma_n$ are indexed by partitions of $n$. For $\lambda \vdash n$, we denote the corresponding irreducible by $V_{\lambda}$. In particular, $V_{(n)}$ is the trivial representation, and $V_{(1^n)}$ is the sign representation.

We will consider the group algebra $K\Sigma_n$ as a subspace of $K[t]\Sigma_n={\mathcal Z}^n_n$.
The group algebra $\K\Sigma_n$ is a $\Sigma_n \times \Sigma_n$-bimodule. The first copy of $\Sigma_n$ acts by composing permutations on the left, and the second copy acts by composing on the right. The action on the right as stated is a right action, but can be converted to a left action via any anti-automorphism such as the inverse map. From the Artin-Wedderburn Theorem, we get:

\begin{proposition} \label{AW}
We have an decomposition of ideals and $\Sigma_n\times\Sigma_n$ representations:
$$\K\Sigma_n = \bigoplus\limits_{\lambda \vdash n} J_\lambda=\bigoplus\limits_{\lambda \vdash n} V_{\lambda} \otimes \widehat{V}_{\lambda}.
$$
where 
$J_\lambda=V_{\lambda} \otimes \widehat{V}_{\lambda}$ is a simple, two-sided ideal, 
and  $\widehat{V}_{\lambda}$ denotes the irreducible representation of the second copy of $\Sigma_n$.
\end{proposition}
The representations $J_\lambda$, $\lambda\vdash n$ of  $\Sigma_n\times \Sigma_n$ are irreducible and pairwise non-isomorphic.

For $\lambda \vdash n$, a standard Young  tableaux $T$ of shape $\lambda$ is a filling of the  Young  diagram corresponding to the partition $\lambda$ with the numbers $1,2,\dots,n$, with increasing rows and columns. The picture below is a standard Young  tableaux of shape $(4,3,3,1)$.

\begin{equation} \label{egtab}
\young(1349,267,58\ten,\eleven) 
\end{equation}

For a standard Young tableau $T$, we define an element in the group algebra $K\Sigma_n$ called the Young symmetrizer. The row stabilizer group $R(T)$ is defined as the subgroup of $\Sigma_n$ that fixed the rows of $T$ and the column stabilizer group $C(T)$ is the subgroup that fixes the columns of $T$.

\begin{definition}
For a standard Young tableaux $T$, we define the Young symmetrizer $y_T$ by the following:
\[ y_T = \sum\limits_{\sigma \in R(T), \mu \in C(T)} \sgn(\mu) [\mu\sigma]\in K\Sigma_n \]
\end{definition}

Young  symmetrizers allow us to concretely understand the isomorphism in Proposition~\ref{AW}. The results we use about them can be summarized in the following proposition.

\begin{proposition}
For $\lambda \vdash n$, let $T$ be a standard Young tableaux of shape $\lambda$. Then we have $$K\Sigma_n \cdot (y_T) \cdot \K\Sigma_n = J_\lambda.$$
\end{proposition}

We have inclusion maps $\Sigma_{n-1} \hookrightarrow \Sigma_n \hookrightarrow \Sigma_{n+1}$. For an irreducible representation $V_{\lambda}$ of $\Sigma_n$, we have a rule for understand the restriction to $\Sigma_{n-1}$ and a rule for understanding the induced representation for $\Sigma_{n+1}$. In the language of symmetric functions, this is often called Pieri's rule.

\begin{proposition} \label{Ind-Res}
Let $\lambda \vdash n$, then we have
\begin{enumerate}
\item ${\rm Res}_{n-1}^n V_{\lambda} \cong \bigoplus\limits_{\zeta = \lambda - \square} V_{\zeta}$;
\item ${\rm Ind}_{n}^{n+1} V_{\lambda} = K\Sigma_{n+1} \otimes_{K\Sigma_n} V_{\lambda} \cong \bigoplus\limits_{\zeta = \lambda \cup \square} V_{\zeta}$.
\end{enumerate}
The notation $\zeta=\lambda-\square$ (resp. $\zeta=\lambda\cup \square$) means that $\zeta$ runs over all partitions obtained from $\lambda$ by deleting (resp. adding)
a box.
\end{proposition}

We use matrix notation to give coordinates to the boxes in a partition. For example, in the standard Young tableaux given above in $(\ref{egtab})$, $7$ is entered in the box with coordinates $(2,3)$. We say $(i,j) \in \lambda$ if $\lambda$ contains the box with coordinates $(i,j)$, which happens precisely when $\lambda_i \geq j$. We say a box is on diagonal $d$ if $j - i = d$. Hence in the Young tableaux in $(\ref{egtab})$, the box containing $7$ is on diagonal $1$.


\subsection{Preliminary results on $\ZZ$}
In this section, we want to understand how the operations of contraction and tensor interact with the action of $\Sigma_n \times \Sigma_n$ in the wheeled PROP $\ZZ$. In order to do so, we will need to use the results on the representation theory of the symmetric groups that we recalled in the previous section.

We have an isotypic decomposition
$$
\ZZ_n^n = K[t]\Sigma_n= \bigoplus\limits_{\lambda \vdash n} \K[t] \otimes J_{\lambda},
$$
where $J_{\lambda} = V_{\lambda} \otimes \widehat{V}_{\lambda}$. 

\begin{lemma} \label{div1}
For a partition $\lambda\vdash n$ we have in $K\Sigma_{n+1}$ that
$$
K\Sigma_{n+1} \cdot (J_{\lambda} \otimes {\rm id}) \cdot K\Sigma_{n+1} = \bigoplus\limits_{\nu = \lambda \cup \square} J_{\nu} 
$$
\end{lemma}

\begin{proof}
This follows from the second part of Proposition~\ref{Ind-Res}.
\end{proof}

\begin{proposition} \label{div2}
The space $\partial_n^n(J_{\lambda})$ is the direct sum of all $(t+i-j)J_{\nu}$ where
$\nu$ is a partition obtained from $\lambda$ by deleting a box at position $(i,j)$.
\end{proposition}

In order to prove the proposition, it will be necessary to do some computations. Let us consider the standard Young tableau of shape $(2,1)$.

$$ 
T =  \young(12,3)$$

Consider the Young  symmetrizer corresponding to the tableaux $y_T = (e - (13))(e + (12))$, i.e., 

\centerline{$y_T=\ \ $\adjustbox{valign=c}{\includegraphics[width=2.5in]{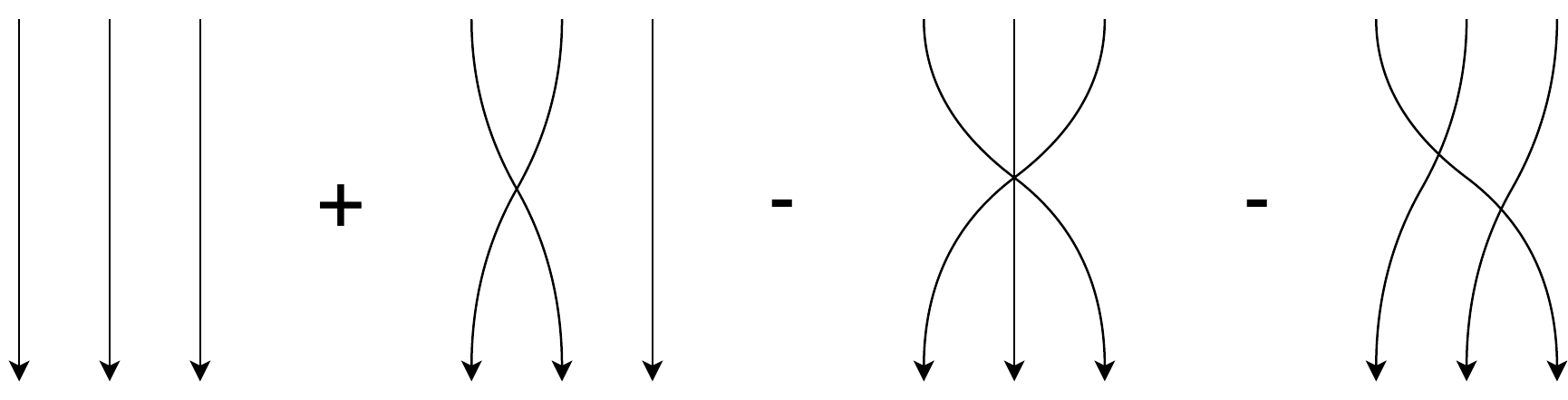}}}

Now, let us apply $\partial_n^n$:

\centerline{$\partial_n^n(y_T)=\ \ $\adjustbox{valign=c}{\includegraphics[width=2.5in]{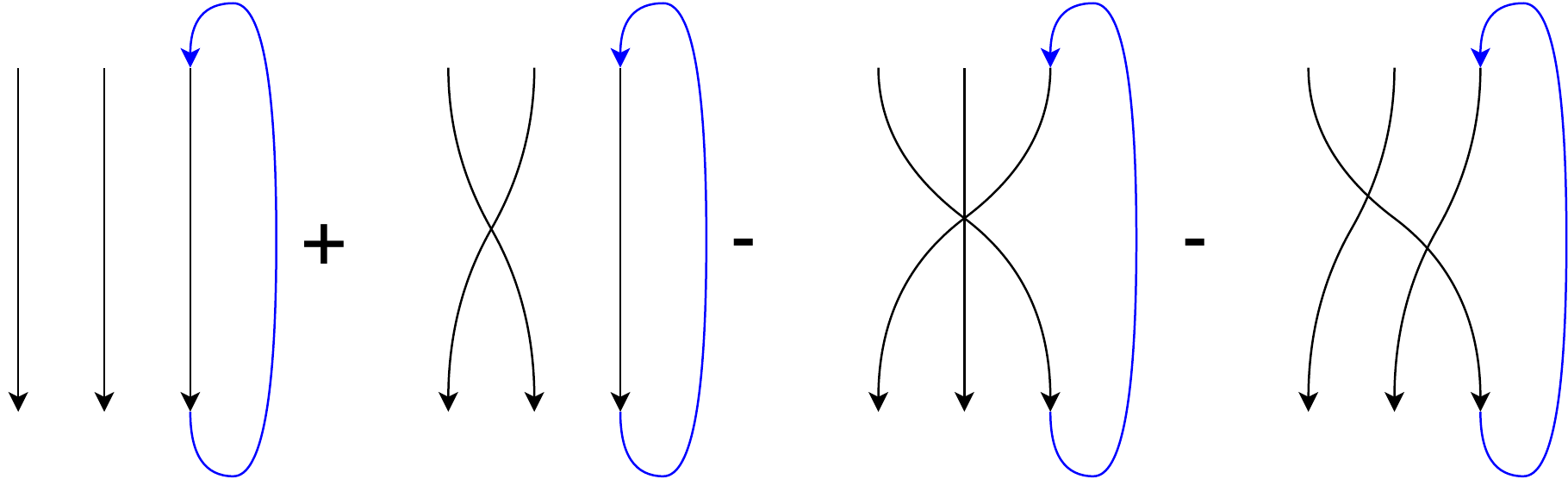}}}

Observe that this is just $(t-1)y_{T'} = (t-1)(e +(12))$, where $T'$ is the tableaux obtained by deleting the box containing $3$. This computation generalizes. In more complicated examples, there is some cancellation as well. We give a complete argument.

\begin{lemma}
Let $T$ be a standard Young tableaux of shape $\lambda$, with $\lambda \vdash n$. Let $T'$ denote the standard Young tableaux obtained by removing the box containing $n$. Then we have $\partial_n^n(y_T) = (t+i-j) y_{T'}$, where $(i,j)$ is the position of the box in the tableau $T$ containing $n$.
\end{lemma}

\begin{proof}
Recall that $$y_T = \sum\limits_{\sigma \in R(T), \mu \in C(T)} \sgn(\mu) [\mu\sigma].$$ 
We need to compute $\partial_n^n(y_T)$. There are four types of terms. We will count the contributions of each of these types independently, and then put them together. Before we do that, we compute $\partial^n_n[\mu\sigma]$ according to several cases.\\

\noindent {\bf Case 1:} $\sigma(n) = n$ and $\mu(n) = n$.

There is a natural $1-1$ correspondence between permutations in $R(T')$ and permutations in $R(T)$ that fix $n$. For $\sigma \in R(T)$ fixing $n$, we will denote its image in $R(T')$ under this correspondence by $\sigma'$. This can be better expressed in terms of our notation as $[\sigma'] \otimes \id = [\sigma]$. In pictures, we have

\centerline{{\includegraphics[width=2in]{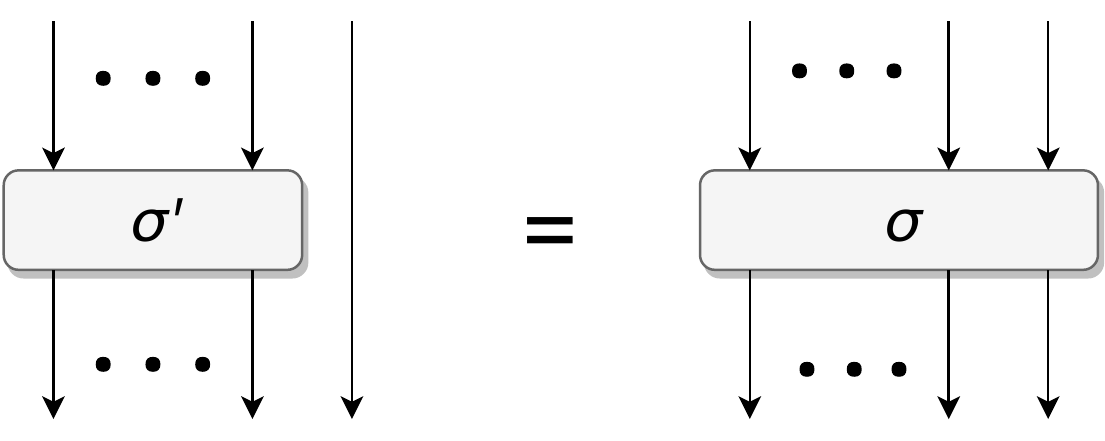}}}

We have a similar correspondence for permutations in $C(T)$ fixing $n$ and permutations in $C(T')$. Since $(\mu\sigma) (n) = n$, when we perform the contraction, we get an exceptional loop, and the rest of the permutation is not disturbed, i.e., $[\mu\sigma] = [\mu'\sigma'] \otimes {\rm id}$. So we have
$$
\partial_n^n [\mu\sigma] =  \partial_n^n ([\mu'\sigma'] \otimes {\rm id})  = t \cdot [\mu'\sigma'] .
$$
In pictures, we have 

\centerline{{\includegraphics[width=2in]{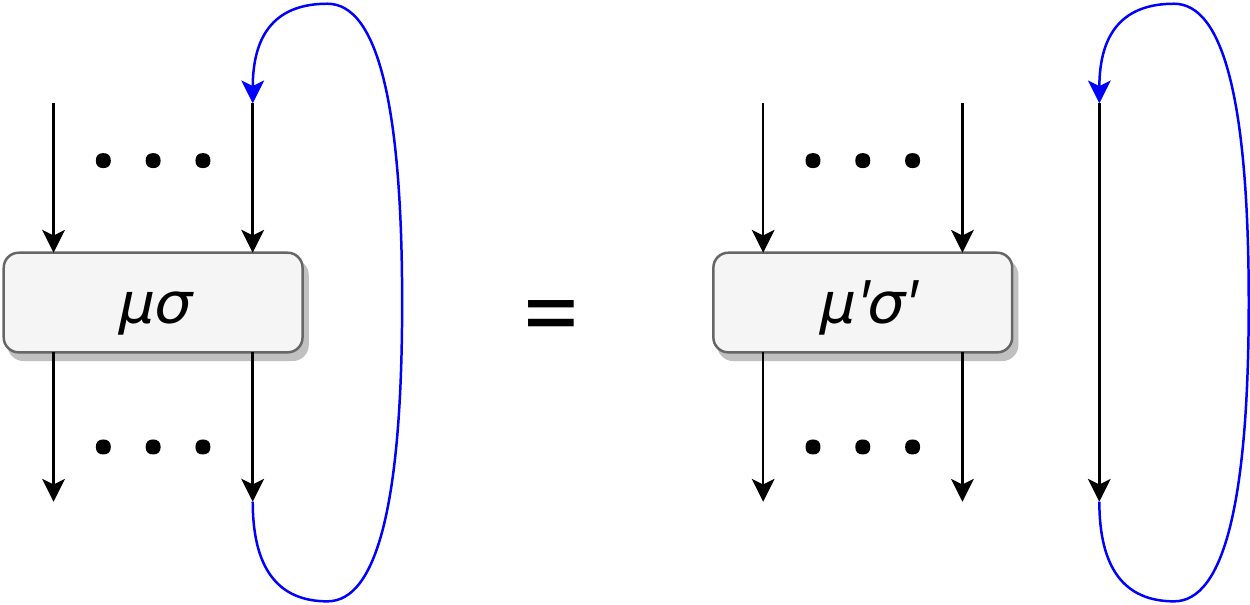}}}

Further, the sign of $\mu$ and $\mu'$ are the same. Thus, the contribution of these terms is $t \cdot y_{T'}$. \\

\noindent {\bf Case 2:} $\sigma(l) = n$ and $\mu(n) = n$ with $l \neq n$.

In order to have $\sigma(l) = n$, we need $l$ to be in the same row as $n$. Now, for each such $l$, there is a $1-1$ correspondence between permutations in $R(T)$ that fix $n$, and permutations in $R(T)$ that send $l$ to $n$, given by $\pi\leftrightarrow \pi\,(l\ n)$.  

So, we can write $\sigma =  \pi\,(l\ n)$ with $\pi \in R(T)$ that fixes $n$. We have $\partial_n^n [\mu\sigma] = \partial_n^n[\mu\pi\,(l\ n)] = [\mu'\pi']$. In pictures, we have

\centerline{$\partial_n^n [\mu\sigma] = \partial_n^n[\mu \pi\,(l\ n)] = \partial_n^n([\mu'\pi'\otimes {\rm id}][(l\ n)])=$\adjustbox{valign=c}{\includegraphics[width=1.5in]{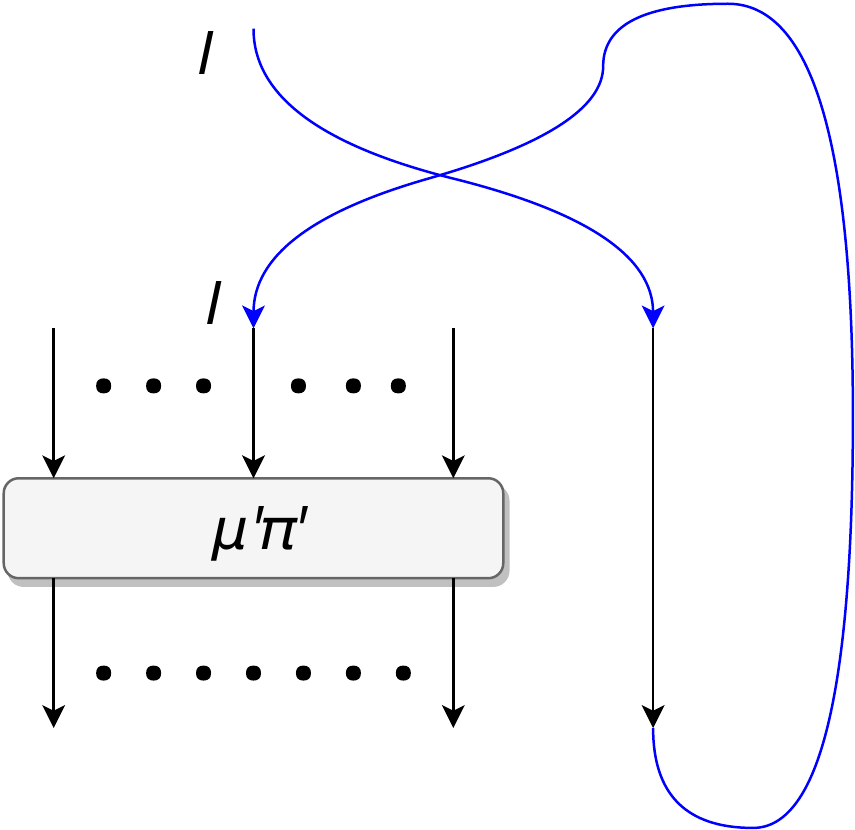}}$\ =[\mu'\pi']$.}

Hence we get a contribution of $y_{T'}$ for each $l\neq n$ in the same row as $n$ in $T$. If $n$ is in position $(i,j)$ of the tableau $T$,
then the contribution is
 is $(i-1) y_{T'}$. \\

\noindent {\bf Case 3:} $\sigma(n) = n$ and $ \mu(n) = k$, with  $k \neq n$.

A similar argument shows that the contribution of these terms $-(j-1) y_{T'}$ if $(i,j)$ is the position of $n$ in the tableau $T$. 
The negative sign comes because the terms $\mu$ and $(k\ n)\,\mu$ have opposite signs. \\

\noindent {\bf Case 4:}  $\sigma(n) = p$ and $\mu(q) = n$ where $p \neq n\neq q$.

The contribution of these terms is $0$. We will give a sign reversing involution of these terms. The number $p$ must be in the same row as $n$, and $q$ must be in the same column as $n$. Let $r$ be the entry in the box that is in the row of $q$ and column of $p$. 
The involution is given by $\mu\sigma \leftrightarrow\widetilde\mu \widetilde\sigma$, where $\widetilde{\sigma} = (p\ r)\,\sigma$, and $\widetilde\mu = \mu\,(q\ r)$. We leave it to the reader to check that this is a sign reversing involution. Further, upon applying $\partial_n^n$, the two terms give the same permutation in $\Sigma_{n-1}$ but with different signs. Thus the contribution of these terms is $0$. \\

Combining the cases above gives $\partial_n^n(y_T) = (t + i-j)(y_T)$.
\end{proof}

\begin{proof} [Proof of Proposition~\ref{div2}]
We have ${\rm Res}^{n}_{n-1} V_{\lambda} = \bigoplus_{\nu = \lambda - \square} V_{\nu}$. Hence 
$${\rm Res}_{n-1}^{n}J_{\lambda} = \bigoplus_{\nu_1=\lambda-\square}\ \bigoplus_{\nu_2=\lambda-\square} V_{\nu_1} \otimes \widehat{V}_{\nu_2}.$$ 
Since $\partial_n^n$ is equivariant with respect to the action of $\Sigma_{n-1} \times \Sigma_{n-1}$, the image of $J_{\lambda}$ can only contain irreducible representations 
of the form $V_{\nu_1} \otimes \widehat{V}_{\nu_2}$ where both $\nu_1$ and $\nu_2$ are obtained by removing a box from $\lambda$,
each with multiplicity at most 1.
On the other hand,  we know $\ZZ^{n-1}_{n-1}$ only contains irreducible representations of the form $V_{\nu} \otimes \widehat{V}_{\nu}$.
So it follows that each of the representations $V_{\nu} \otimes \widehat{V}_{\nu}$ can appear in $\partial_n^n(J_\lambda)$ with
multiplicity at most 1. 

Let $y_{T}$ be the Young symmetrizer corresponding to a standard Young tableau of shape $\lambda$ and suppose that the box of $T$ containing $n$ is in position $(i,j)$. Then we have $\partial_n^n(y_T) = (t+ i-j)y_{T'}$ where $T'$ is the tableau obtained by removing the box containing $n$ as shown in the above lemma. Thus we have $(t+i-j) J_{\nu} \subseteq \partial^n_n(J_\lambda)$, where $\nu$ is the partition obtained by removing the box containing $n$. 

For any $\nu$ that can be obtained from $\lambda$ by removing a box, we can take a standard Young tableau $T'$ of shape $\nu$, and consider a standard Young tableaux $T$ of shape $\lambda$ obtained by putting $n$ in the box that was removed. The above argument then shows that $(t+i-j) J_{\nu}\subseteq \partial^n_n(J_\lambda)$.

We have already seen that each $J_\nu=V_\nu\otimes \widehat{V}_\nu$ can appear in $\partial^n_n(J_\lambda)$ with multiplicity at most 1,
so $\partial^n_n(J_\lambda)$ is the sum of all spaces $(t+i-j)J_\nu$, where $\nu$ is obtained from $\lambda$ by removing a box at position $(i,j)$.
%
%
%
\end{proof}

\subsection{Ideals of $\ZZ$}
The aim of this section is to classify all the ideals of $\ZZ$. Let ${\mathcal I}\subseteq \ZZ$ be an ideal. Clearly, we have ${\mathcal I}^m_n= 0$ for $m \neq n$. For a polynomial $f \in K[t]$, we denote the ideal generated by $f$ by $(f)$.

\begin{lemma}
We have ${\mathcal I}^n_n = \bigoplus\limits_{\lambda \vdash n} (g_{\lambda}) \otimes J_{\lambda}$, where $g_{\lambda}$ is either a monic polynomial or $0$.
\end{lemma}

\begin{proof}
The space ${\mathcal I}^n_n$ is a $\Sigma_n \times \Sigma_n$-stable subset of $\ZZ_n^n = \K[t][\Sigma_n]$. Further ${\mathcal I}^n_n$ is stable under multiplication by $\ZZ^0_0 = K[t]$. Since $\ZZ_n^n = \bigoplus\limits_{\lambda \vdash n} \K[t] \otimes J_{\lambda}$, we have  ${\mathcal I}^n_n = \bigoplus\limits_{\lambda \vdash n} L_{\lambda} \otimes J_{\lambda}$, where $L_{\lambda}$ is an ideal of $K[t]$. 
If $L_\lambda\neq 0$ then there is a unique monic polynomial $g_{\lambda}$ such that $L_{\lambda} = (g_{\lambda})$. 
\end{proof}

We let $\emptyset$ denote the empty partition. By definition $g_\emptyset$ generates ${\mathcal I}^0_0\subseteq K[t]={\mathcal Z}^0_0$ and $g_\emptyset=0$ or $g_\emptyset$ is a monic polynomial.

\begin{lemma} \label{00not0}
Let ${\mathcal I}$ be a non-zero ideal of $\ZZ$, then $g_{\emptyset} \neq 0$. 
\end{lemma}

\begin{proof}
Since ${\mathcal I}$ is a non-zero ideal, we have $0 \neq a \in {\mathcal I}^d_d$ for some $d$. Write $a = \sum_{\sigma \in \Sigma_d} f_{\sigma} [\sigma]$, with $f_{\sigma} \in K[t]$. Without loss of generality, we can assume $f_{\rm id} = 1$. (If not, we can pick $\mu$ such that $f_\mu \neq 0$ and consider instead $\frac{1}{f_\mu}(\mu^{-1} \cdot a) \in {\mathcal I}^d_d$.)

Observe that $\partial_1^1  \partial_2^2 \cdots \partial_n^n (a) = t^n +$ lower order terms, giving a non-zero element in ${\mathcal I}^0_0$.
\end{proof}

\begin{corollary} \label{nonzero}
If ${\mathcal I}$ is a non-zero ideal, we have that $g_{\emptyset} \in (g_{\lambda})$. In particular, 
$g_{\lambda} \neq 0$ for all $\lambda$. 
\end{corollary}

\begin{proof}
We have $g_{\emptyset} \neq 0$. Thus $g_{\emptyset} \cdot J_{\lambda} \in {\mathcal I}$, and hence $g_{\emptyset} \in (g_{\lambda})$.
\end{proof}

\begin{definition}
A collection of monic polynomials $\{q_{\lambda}\}$ indexed by partitions $\lambda$ is called compatible if for any two partitions $\lambda$ and $\mu$ such that $\lambda = \mu \cup \text{one box}$, with the additional box on diagonal $d$, then either $q_{\mu} = q_{\lambda}$ or $q_{\mu} = q_{\lambda} \cdot (t + d)$.
\end{definition}

Given an ideal ${\mathcal I}$ we have ${\mathcal I}^n_n = \bigoplus\limits_{\lambda \vdash n} (g_{\lambda}) \otimes J_{\lambda}$. This gives us a collection of polynomials $\{g_{\lambda}\}$.

\begin{theorem} \label{divisibility}
Given an ideal ${\mathcal I}$, the collection of polynomials $\{g_{\lambda}\}$ is compatible. Conversely, given a compatible collection $q_{\lambda}$, we can define an ideal ${\mathcal J}$ by ${\mathcal J}^n_n = \bigoplus\limits_{\lambda \vdash n} (q_{\lambda}) \otimes J_{\lambda}$.
\end{theorem}

\begin{proof} 
By definition, a subset ${\mathcal I}= \bigoplus {\mathcal I}^m_n$ is an ideal if and only if each ${\mathcal I}^m_n$ is stable under the action of $\Sigma_m \times \Sigma_n$, and is closed under contractions and the tensor product. In fact, since $\ZZ$ is generated by the identity, it suffices to consider tensoring by ${\rm id}$.  Hence, the theorem follows directly from Lemma~\ref{div1} and Proposition~\ref{div2}.
\end{proof}

Thus it is suffices to classify all the compatible collections of monic polynomials. Given a finite collection of boxes $\{(i_1,j_1),(i_2,j_2) \dots (i_k,j_k)\}$ and a monic polynomial $f$, we define a compatible collection $\{g_{\lambda}\}$ given by $g_\lambda = f \cdot \prod\limits_{p = 1}^k (t + d_ p)^{\epsilon_p}$, where $d_p = j_p - i_p$, and $\epsilon_p = 0$ if $(i_p,j_p) \in \lambda$, and $1$ otherwise. Since the collection of boxes is finite, it is contained in some large enough rectangular partition. Hence, we can pictorially show the collection as a rectangular partition with some boxes shaded in. For example, if we have the collection $\{(1,1), (1,3),(4,2)\}$, this is contained in the rectangular partition $(3,3,3,3)$, and the picture is given as follows:

\[
\begin{tikzpicture} 
    [
        box/.style={rectangle,draw=black,thick, minimum size=0.5cm},
    ]

\foreach \x in {1,2,3}{
    \foreach \y in {-1,-2,-3,-4}
        \node[box] at (\x/2,\y/2){};
}

\node[box,fill=gray] at (1/2,-1/2){};  
\node[box,fill=gray] at (3/2,-1/2){};  
\node[box,fill=gray] at (2/2,-4/2){};  

\end{tikzpicture}
\]

In fact, we will show that all collections of compatible polynomials are of this form. Let $\{g_{\lambda}\}$ be a compatible collection of monic polynomials. We say that $\lambda$ has an $(i,j)$-jump if $(i,j)$ is a removable box, and $g_{\lambda} \neq g_{\lambda - (i,j)}$. The compatibility condition forces $g_{\lambda - (i,j)} = (t+d)g_{\lambda}$ where $d = j-i$.

\begin{lemma}
Suppose $\lambda_1$ and $\lambda_2$ are partitions such that $(i,j)$ is a removable box in both of them. Then either both have an $(i,j)$-jump or neither have an $(i,j)$-jump.
\end{lemma}

\begin{proof}
We can assume $\lambda_2$ is the rectangular partition $j^i = \underbrace{(j,j,\dots,j)}_{i \text{ times}}$. Observe that we have $\lambda_2 \subset \lambda_1$.To go from $g_{\lambda_1}$ to $g_{\lambda_2 - (i,j)}$, we can remove removable boxes from $\lambda_1$ and track the jumps. 

By removing boxes from $\lambda_1$ we can reach $\lambda_2 - (i,j)$ in two ways. So, we can get to $g_{\lambda_2 - (i,j)}$ from $g_{\lambda_1}$ in two ways as shown in the picture below.

\[\begin{tikzcd} 
g_{\lambda_1} \arrow{r}{} \arrow[swap]{d}{} & g_{\lambda_2} \arrow{d}{} \\
g_{\lambda_1 - (i,j)} \arrow{r}{} & g_{\lambda_2 - (i,j)}
\end{tikzcd}
\]

The horizontal transformations come out of removing boxes from the skew partition $\lambda_1 \smallsetminus \lambda_2$, and hence none of the jumps affect the multiplicity of $t+d$, where $d = j-i$. This is because none of the boxes in the skew partition are on diagonal $d$. Thus either both vertical arrows are jumps or neither are jumps.
\end{proof}

The above lemma shows that jumps are determined not by the partitions themselves, but rather the position of the removable boxes. Since $g_{\emptyset}$ is a polynomial of a finite degree, we can only have a finite number of jumps. Thus collections of compatible monic polynomials are indexed by $(f,\mathcal{C})$ where $f \in K[t]$ is a monic polynomial, and $\mathcal{C}$ is a finite subset of $\Z_{> 0}^2$. 

\begin{definition}
Given a monic polynomial $f \in K[t]$ and a finite subset $\mathcal{C} = \{(i_1,j_1),\dots,(i_k,j_k) \}\subset \Z_{>0}^2$, let $g_{\lambda} = f \cdot \prod_{p=1}^k (t+d_p)^{\epsilon_p}$ where $d_p = j_p - i_p$, and $\epsilon_p = 0$ if $(i_p,j_p) \in \lambda$ and $\epsilon_p = 1$ otherwise. We define an ideal $ {\mathcal I}= \bigoplus_{p \in \Z_{\geq 0}} {\mathcal I}^p_p$, where 
$$
{\mathcal I}^p_p = \bigoplus\limits_{\lambda \vdash n} (g_{\lambda}) \otimes J_{\lambda}.
$$
Note in particular that $g_{\emptyset} =  f \cdot \prod_{p=1}^k (t+d_p)$.
\end{definition}

We summarize the above discussion:

\begin{theorem} \label{ideals}
Every nonzero ideal of $\ZZ$ is of the form ${\mathcal I}(f,\mathcal{C})$ for some monic polynomial $f \in K[t]$, and $\mathcal{C}$ is a finite subset of $\Z_{>0}^2$.
\end{theorem}

Note that the monic polynomial $f$ could be $1$.

\subsection{Prime ideals and maximal ideals} 
We will now classify the prime ideals and maximal ideals of ${\mathcal Z}$.

\begin{definition}
An ideal ${\mathcal I}$ of a wheeled PROP $\mathcal{A}$ is a prime ideal if for $a \in {\mathcal A}^p_q$ and $b \in {\mathcal A}^r_s$, we have $a \otimes b \in {\mathcal I} \implies a \in {\mathcal I}$ or $b \in {\mathcal I}$. The ideal ${\mathcal I}$ is called maximal if
${\mathcal I}\neq {\mathcal A}$ and for every ideal ${\mathcal J}$ with ${\mathcal I}\subseteq {\mathcal J}\subseteq {\mathcal A}$
we have ${\mathcal J}={\mathcal I}$ or ${\mathcal J}={\mathcal A}$. 
\end{definition}
\begin{lemma}
A maximal ideal of ${\mathcal A}$ is prime.
\end{lemma}
\begin{proof}
Suppose that ${\mathcal M}$ is a maximal ideal and $a\otimes b\in {\mathcal M}$ where $a\in {\mathcal A}^p_q$ and $b\in {\mathcal A}^r_s$. 
Assume that $a\not\in {\mathcal M}$. Then $1$ lies in the ideal generated by $a$ and ${\mathcal M}$
and $1\otimes b$ lies in the ideal generated by $a\otimes b$ and ${\mathcal M}\otimes b$. This shows that $b=1\otimes b\in {\mathcal M}$.
\end{proof}

\begin{remark}
For a prime ideal $\PR$ of a wheeled PROP $\mathcal{A}$, we have  that $\PR^0_0$ is a prime ideal of the commutative ring $\mathcal{A}^0_0$. 
\end{remark}

Let $\PR$ be a prime ideal in the wheeled PROP $\ZZ$, and let $\PR^n_n = \bigoplus\limits_{\lambda \vdash n} (g_{\lambda}) \otimes J_{\lambda}$, where the $g_{\lambda}'s$ are monic polynomials or $zero$.



\begin{proposition}
The prime ideals of $\ZZ$ are $0$,  ${\mathcal I}(t-a, \emptyset)$ for some $a\in K$ or ${\mathcal I}(1,\{(i,j)\})$ for some $i,j\in \Z_{>0}$.
\end{proposition}
\begin{proof}
Suppose that $\PR$ is a prime ideal of $\ZZ$.
By the Remark above, $\PR^0_0$ is a prime ideal of $\ZZ^0_0=K[t]$,
so it is equal to $0$ or to $(t-a)$ for some $a\in K$. If $\PR^0_0=0$, then we get $\PR=0$.
Suppose that $\PR^0_0=(t-a)$.  We have $\PR={\mathcal I}(f,{\mathcal C})$ where $f\in K[t]$ is monic and ${\mathcal C}$ is finite.
and $t-a=f\prod_{(i,j)\in {\mathcal C}} (t+j-i)$. There are 2 cases:
\begin{enumerate}
\item  $f=t-a$ and ${\mathcal C}=\emptyset$, or
\item $f=1$, ${\mathcal C}=\{(i,j)\}$ and $a=i-j$.
\end{enumerate}
\end{proof}
\begin{corollary}
The maximal ideals of $\ZZ$ are ${\mathcal I}(t-a,\emptyset) $ with $a\in K\setminus\Z$ and ${\mathcal I}(1,\{(i,j)\})$ with $i,j\in \Z_{>0}$.
\end{corollary}
\begin{proof}
Maximal ideals are exactly the prime ideals that are not properly contained in other prime ideals. Now the corollary follows from the classification of prime ideals.
\end{proof}

\section{Representable wheeled PROPs}\label{sec5}
\subsection{An equivalence of categories}

Suppose that $G$ is an affine algebraic group (over $K$) and $K[G]$ is the coordinate ring of $G$ and $W$ is a $K$-vector space. 
An action of $G$ on $W$ is called rational if there exists a $K$-linear map $\gamma:W\to K[G]\otimes W$
with the following property: if $\gamma(w)=\sum_{j=1}^r f_j\otimes w_j$ for some $r$, functions $f_1,\dots,f_r\in K[G]$ and vectors $w_1,\dots,w_r\in W$, then $g\cdot w=\sum_{j=1}^r f_j(g)w_j$ for all $g\in G$. 
 A $G$-algebra is a commutative $K$-algebra $R$ with a rational action of $G$ such
 that $G$ acts by $K$-algebra automorphisms.

For this section, let $V$ denote an $n$-dimensional vector space, and let $\mathcal{V} = \bigoplus_{p,q \in \N} \V^p_q$ denote the mixed tensor algebra over $V$. We introduce two categories:
\begin{description}
\item[$\GLn$] 
 In the category $\GLn$ the objects are commutative $\GL(V)$-algebras and morphisms are $\GL(V)$-equivariant $K$-algebra homomorphisms.

\item[$\BBn$] A wheeled PROP is called $n$-representable if it is isomorphic to a sub-wheeled PROP of $R\otimes {\mathcal V}$ where $R$ is a commutative $K$-algebra. The objects of $\BBn$ are $n$-representable wheeled PROPs. 
The morphisms in $\BBn$ are homomorphisms of wheeled PROPs.
\end{description}

We construct a {\it covariant functor} $\Theta:\GLn\to\BBn$ as
follows. Suppose that $R$ is a $\GL(V)$-algebra. Note that $\GL(V)$
acts on $R$ as well as on the algebra ${\mathcal V}$. We define
$$\Theta(R)=(R\otimes {\mathcal V})^{\GL(V)}=\bigoplus_{p,q\in \N} (R\otimes
\V^p_q)^{\GL(V)}.
$$
It is easy to see that $(R\otimes {\mathcal V})^{\GL(V)}$ is closed
under $\otimes$ and $\partial^i_j$. So $\Theta(R)=(R\otimes
{\mathcal V})^{\GL(V)}$ is a subalgebra of $R\otimes {\mathcal V}$, and hence an object of $\BBn$.

Suppose that $\psi:R\to S$ is a $\GL(V)$-equivariant ring
homomorphism. Then $\psi$ induces a $\GL(V)$-equivariant
homomorphism of wheeled PROPs
$$
\psi \otimes {\rm id} :R\otimes {\mathcal V}\longrightarrow S\otimes
{\mathcal V}.
$$
Taking $\GL(V)$-invariants on both sides gives us a homomorphism of
wheeled PROPs
$$
\Theta(R)=(R\otimes {\mathcal V})^{\GL(V)}\longrightarrow
\Theta(S)=(S\otimes{\mathcal V})^{\GL(V)}
$$
which we will denote by $\Theta(\psi)$. We leave it to the reader to
check that $\Theta$ is indeed a functor.

Before we define a functor in the other direction, we need the
following lemma.
\begin{lem}\label{lem1}
For an object  ${\mathcal A}$ in $\BBn$, there exists a
unique commutative $K$-algebra $R$ and a $K$-algebra homomorphism
$\rho:{\mathcal A}\to R\otimes{\mathcal V}$ of wheeled PROPs with the following universal property: If $S$ is a
commutative $K$-algebra and $\lambda:{\mathcal A}\to S\otimes
{\mathcal V}$ is a homomorphism of wheeled PROPs, then
there exists a unique $K$-algebra homomorphism $\psi:R\to S$ such that
$(\psi\otimes \id)\circ \rho=\lambda$. In diagrams:
$$
\xymatrix{{\mathcal A}\ar[r]^-\rho\ar[d]_-{\lambda} & R\otimes {\mathcal V}\ar@{.>}[ld]^-{\psi\otimes \id}\\
S\otimes{\mathcal V} & }.
$$
\end{lem}
\begin{proof}
Suppose that $\rho_i:{\mathcal A}\to R_i\otimes {\mathcal V}, i\in
I$ are, up to isomorphism, all homomorphisms of wheeled PROPs where $R_i$ is a commutative $K$-algebra whose cardinality
is at most the cardinality of ${\mathcal A}$ (this to ensure that
$I$ is still a set). Define
$$\rho=\prod_{i\in I}\rho_i:{\mathcal A}\to (\prod_{i\in
I}R_i)\otimes{\mathcal V}.$$ Define $R$ as the smallest subring of
$\prod_{i\in I}R_i$ such that $\rho({\mathcal A})\subseteq R\otimes
{\mathcal V}$. One can now easily show that $\rho:{\mathcal A}\to
R\otimes {\mathcal V}$ has the desired universal property. The
uniqueness of $R$ follows from the universal property, as usual.
\end{proof}

We define a functor $\Theta:\BBn\to\GLn$ as follows. For an object ${\mathcal A}$ in $\BBn$, define
$\Theta({\mathcal A})=R$, where $\rho:{\mathcal A}\to R\otimes
{\mathcal V}$ is as in Lemma~\ref{lem1}. We can define a regular
$\GL(V)$-action on $R$ as follows. The action of $\GL(V)$ on
${\mathcal V}$ corresponds to a map
$$
\gamma:{\mathcal V}\to K[\GL(V)]\otimes {\mathcal V}.
$$
If we tensor with $R$ we get a map
$$
\id\otimes \gamma:R\otimes {\mathcal V}\to R\otimes
K[\GL(V)]\otimes{\mathcal V}.
$$
The composition
$$
(\id\otimes \gamma)\circ \rho:{\mathcal A}\to R\otimes
K[\GL(V)]\otimes{\mathcal V}
$$
is a homomorphism of wheeled PROPs. The universal property
of $\rho:{\mathcal A}\to R\otimes {\mathcal V}$ implies that there
is a unique homomorphism $\mu:R\to R\otimes K[\GL(V)]$ such that
$(\id\otimes \gamma)\circ \rho=(\mu\otimes \id)\circ\rho$:
$$
\xymatrix{
{\mathcal A}\ar[r]^-\rho\ar[d]_-\rho & R\otimes {\mathcal V}\ar[d]^{\id\otimes \gamma} \\
R\otimes {\mathcal V}\ar@{.>}[r]_-{\psi\otimes \id} & R\otimes K[\GL(V)]\otimes {\mathcal V}}.
$$

The reader may check, using universality, that $\mu:R\to
R\otimes K[\GL(V)]\otimes R$ defines a regular {\it right\/} action of
$\GL(V)$ on $R$. In other words, if $\mu(f)=\sum_{i} f_i\otimes h_i$
then
$$
f\cdot g:= \sum_i f_ih_i(g),\quad g\in \GL(V)
$$
defines a right action of $\GL(V)$ on $R$. Of course, we may view
$R$ as an algebra with a regular {\it left\/} $\GL(V)$ action by
defining
$$
g\cdot f:=f\cdot g^{-1}.
$$
This shows that $R$ is a $\GL(V)$-algebra. Left and right
multiplication (on $R$ or on ${\mathcal V}$) by $g\in \GL(n)$ shall
be denoted by $L_g$ and $R_g$ respectively. The above reasoning
shows that
$$
(L_{g^{-1}}\otimes\id)\circ \rho=(R_g\otimes \id)\circ
\rho=(\id\otimes L_g)\circ \rho:{\mathcal A}\to R\otimes {\mathcal
V}
$$
and therefore
$$
(L_{g}\otimes L_g)\circ \rho=\rho.
$$
This shows that $\rho({\mathcal A})\subseteq (R\otimes {\mathcal
V})^{\GL(V)}$, where $\GL(V)$ acts on the left on $R$ and ${\mathcal
V}$.

Suppose that $\phi:{\mathcal A}\to {\mathcal B}$ is a homomorphism
of $n$-representable wheeled PROPs. Let $\rho_{A}:{\mathcal A}\to R\otimes
{\mathcal V}$ and $\rho_{B}:{\mathcal B}\to S\otimes {\mathcal V}$
be the universal maps as in Lemma~\ref{lem1}. Consider the
composition $\rho_B\circ \phi:{\mathcal A}\to S\otimes {\mathcal
V}$. There exists a unique ring homomorphism $\psi:R\to S$ such that
$(\psi\otimes \id)\circ \rho_A=\rho_B\circ \phi$:
$$
\xymatrix{{\mathcal A}\ar[r]^-{\rho_A}\ar[d]_{\phi} & R\otimes {\mathcal V} \ar@{.>}[d]^{\psi\otimes\id}\\
{\mathcal B}\ar[r]_-{\rho_B} & S\otimes {\mathcal V}}.
$$

Define $\Theta(\phi)=\psi$. The reader may check that $\Theta$
defines a functor.

\begin{thm}
$\Phi$ and $\Theta$ are each other inverses, i.e., the categories
$\GLn$ and $\BBn$ are equivalent.
\end{thm}
\begin{proof}
Suppose that ${\mathcal A}=\bigoplus_{p,q\in \N} {\mathcal A}^p_q$ is an $n$-representable wheeled PROP. So there exists an injective homomorphism
of wheeled PROPs $\lambda:{\mathcal A}\to S\otimes
{\mathcal V}$ for some commutative $K$-algebra $S$.
 Let $\rho:{\mathcal A}\to R\otimes {\mathcal V}$ be as
in Lemma~\ref{lem1}. By the universal property there exists a
$K$-algebra homomorphism $\mu:R\to S$ such that $(\mu\otimes
\id)\circ \rho=\lambda$. Since $\lambda$ is injective, $\rho$ is
injective. We have seen that $\rho({\mathcal A})\subseteq (R\otimes
{\mathcal V})^{\GL(V)}$. We claim that equality holds.  Let $T$ be
the subring generated by all
$\langle \rho(a),v\rangle$
where $a\in {\mathcal A}^p_q$ for some $p,q\in \N$, $v\in \V^q_p$ and
$\langle\cdot,\cdot\rangle$ is the bilinear pairing $\V^p_q\times
\V^q_p\to K$ which naturally extends to a pairing $(R\otimes
\V^p_q)\times \V^q_p\to R$. Clearly
$$
\langle \rho(a_1\otimes a_2),v_1\otimes v_2\rangle= \langle
\rho(a_1),v_1\rangle \langle \rho(a_2),v_2\rangle$$ so $T$ is the
{\it $K$-vector space\/} spanned by all $\langle \rho(a),v\rangle$,
with $a\in {\mathcal A}^p_q$ and $v\in \V^q_p$ and $p,q\in \N$. We have
$\rho({\mathcal A})\subseteq T\otimes {\mathcal V}$. From the
universal property  of $\rho:{\mathcal A}\to R\otimes {\mathcal V}$ follows that $T=R$.

Suppose that $u\in (R\otimes \V^p_q)^{\GL(V)}$ for some $p,q\in \N$.
Then there exist $a_i\in {\mathcal A}^{p_i}_{q_i}$, $v_i\in {\mathcal V}^{q_i}_{p_i}$ and $w_i\in \V^p_q$
such that $u=\sum_{i=1}^r\langle \rho(a_i),v_i\rangle w_i$.
Let $f_i=v_i\otimes w_i\in
\Hom(\V^{p_i}_{q_i},\V^p_q)$. Then we have
$$
u=\sum_{i=1}^r f_i(\rho(a_i)).
$$
Consider the action of $\GL(V)$. The elements $u$, and $\rho(a_i)$,
$i=1,2,\dots,r$ are $\GL(V)$-invariant, but $f_1,\dots,f_r$ may not
be. By applying the Reynolds operator on both sides, we may assume
that $f_1,\dots,f_r$ are $\GL(V)$-invariant as well. Now $f_i$ is a
$\GL(V)$-invariant tensors in
$$
\Hom(\V^{p_i}_{q_i},\V^p_q)\cong \V^{q_i}_{p_i}\otimes \V^p_q\cong
\V^{p+q_i}_{q+p_i}.
$$
The first fundamental theorem in invariant theory for $\GL(V)$ tells
us exacly how such a tensor looks like. This means that $f_i$ as a
linear map, is a composition of contractions and tensoring with the
identity in $V^\star\otimes V$. But $\rho({\mathcal A})$ is closed
under contractions and tensoring with the identity. This shows that
$f_i(\rho(a_i))$ lies in the image of $\rho$. But then $u$ lies in
the image of $\rho$. So $\rho$ defines an isomorphism between
${\mathcal A}$ and $(R\otimes {\mathcal
V})^{\GL(V)}=\Phi(\Theta({\mathcal A}))$. We leave it to the reader
to show that $\Phi\circ\Theta$ is naturally equivalent to identity
functor.

Suppose that $R$ is a $\GL(V)$-algebra.  Let ${\mathcal
A}=\Phi(R)=(R\otimes {\mathcal V})^{\GL(V)}$ and let
$\widetilde{R}=\Theta({\mathcal A})=\Theta(\Phi(R))$. Define
$\rho:{\mathcal A}\to \widetilde{R}\otimes {\mathcal V}$
as in Lemma~\ref{lem1}. Consider the inclusion $\lambda:(R\otimes
{\mathcal V})^{\GL(V)}\to R\otimes {\mathcal V}$. The universal
property implies that there exists a $K$-algebra homomorphism
$\psi:\widetilde{R}\to R$ such that $(\psi\otimes\id)\circ
\rho=\lambda$. Now $\psi\otimes\id:\widetilde{R}\otimes {\mathcal
V}\to R\otimes {\mathcal V}$ restricts to a map of wheeled PROPs
$$
\psi':(\widetilde{R}\otimes {\mathcal V})^{\GL(V)}\to (R\otimes
{\mathcal V})^{\GL(V)}={\mathcal A}
$$
On the other hand, we have seen before that $\rho:{\mathcal A}\to
(\widetilde{R}\otimes {\mathcal V})$ induces an isomorphism
$\rho':{\mathcal A}\to (\widetilde{R}\otimes {\mathcal
V})^{\GL(V)}$. We have $(\psi\otimes\id)\circ\rho=\lambda$ and if we
restrict the codomain to $(R\otimes {\mathcal V})^{\GL(V)}$, then we
get $\psi'\circ \rho'=\id$. Since $\rho'$ is an isomorphism, so is
$\psi'$.
 Since all
irreducible $\GL(V)$-representations appear in the $\V^p_q$, we must
have that $\psi$ is already an isomorphism. So $\Theta(\Psi(R))\cong
R$. Again, we leave it to the reader to verify that the composition
functor $\Theta\circ \Psi$ is naturally equivalent to the identity.

\end{proof}


\begin{prop}\label{prop1}
Suppose $R$ is a $\GL(V)$-algebra and ${\mathcal
A}=\Phi(R)=(R\otimes {\mathcal V})^{\GL(V)}$. Let $\mathcal{I}$ be an ideal of $\mathcal{A}$. Then there exists a
$\GL(V)$-stable ideal $I$ of $R$ such that ${\mathcal I}=(I\otimes
\V)^{\GL(V)}$.
\end{prop}

The previous proposition shows there is a bijection between
$\GL(V)$-stable ideals of $R$ and ideals of ${\mathcal A}=(R\otimes
V)^{\GL(V)}$. If ${\mathcal A}$ is a wheeled PROP and
${\mathcal I}\subseteq {\mathcal A}$ is an ideal, then ${\mathcal
A}/{\mathcal I}$ has a natural structure of a wheeled PROP
because it can be identified with the image of some homomorphism
$\phi:{\mathcal A}\to {\mathcal B}$ of a wheeled PROP.

\begin{remark}
If ${\mathcal A}=(R\otimes {\mathcal V})^{\GL(V)}$ is an $n$-representable wheeled PROP, and ${\mathcal I}\subseteq {\mathcal A}$ is an
ideal, then ${\mathcal A}/{\mathcal I}$ is a wheeled PROP.
Proposition~\ref{prop1} implies that ${\mathcal I}=(I\otimes
{\mathcal V})^{\GL(V)}$ for some $\GL(V)$-stable ideal $I$. If we
apply $\Phi$ to $R/I$ we get
$$
((R/I)\otimes {\mathcal V})^{\GL(V)}\cong (R\otimes {\mathcal
V})^{\GL(V)}/(I\otimes {\mathcal V})^{\GL(V)}={\mathcal A}/{\mathcal
I}.
$$
This shows that ${\mathcal A}/{\mathcal I}\cong \Theta(R/I)$ is an $n$-representable wheeled PROP as well. So $\BBn$ is closed under homomorphic images.
\end{remark}

\begin{proof}[Proof of Proposition~\ref{prop1}]
Let $I\subseteq R$ be the smallest ideal such that
$$
{\mathcal I}\subseteq (I\otimes R)^{\GL(V)}.
$$
We will prove that equality holds.

Suppose that $u\in (I\otimes \V^p_q)^{\GL(V)}$ for some $p,q\in \N$.
Then there exist $a_i\in {\mathcal I}^{p_i}_{q_i}$,
 $f_i\in
\Hom(\V^{p_i}_{q_i},\V^p_q)$ for $i=1,2,\dots,r$ such that
$$
u=\sum_{i=1}^r f_i(a_i)
$$
in $I\otimes {\mathcal V}$.
The elements $u$, and $a_i$,
$i=1,2,\dots,r$ are $\GL(V)$-invariant, but $f_1,\dots,f_r$ may not
be. By applying the Reynolds operator on both sides, we may assume
that $f_1,\dots,f_r$ are $\GL(V)$-invariant as well. Now $f_i$ is a
$\GL(V)$-invariant tensors in
$$
\Hom(\V^{p_i}_{q_i},\V^p_q)\cong \V^{q_i}_{p_i}\otimes \V^p_q\cong
\V^{p+q_i}_{q+p_i}.
$$
Again, the first fundamental theorem in invariant theory for
$\GL(V)$ implies that $f_i$, as a linear map, is a composition of
contractions and tensoring with the identity in $V^\star\otimes V$.
The ideal ${\mathcal I}$ is closed under contractions and tensoring
with the identity. This shows that $f_i(a_i)$ lies in ${\mathcal I}$
for all $i$. So $u$ lies in ${\mathcal I}$.

\end{proof}

\section{Subalgebras of the mixed tensor algebra}\label{sec6}
For $K = \C$, the mixed tensor algebra was studied by Schrijver in \cite{Schrijver08}, where he shows that a subset $A \subseteq \V$ is of the form $\V^G$ for some subgroup $G$ of the unitary group if and only if $A$ is a contraction closed nondegenerate graded subalgebra of $\V$, closed under $\ast$. 

We prove a generalization of this over an arbitrary field of characteristic $0$. 

\begin{definition}
A wheeled PROP $\mathcal{A}$ is simple if it has exactly $2$ ideals, namely the zero ideal and $\mathcal{A}$ itself.
\end{definition} 
Suppose ${\mathcal M}$ is an ideal of ${\mathcal A}$.
It is clear from the definitions that  ${\mathcal A}/{\mathcal M}$ is simple if and only if ${\mathcal M}$ is a maximal ideal.
In particular, ${\mathcal A}$ is simple if and only if the zero ideal is maximal.

\begin{theorem} \label{Schrijver-generalization}
There is a bijection between simple subalgebras of ${\mathcal V}$
and Zariski closed reductive subgroups $G\subseteq \GL(V)$ which are
defined over $K$.
\end{theorem}

We refer the reader to \cite{Humphreys, Springer} for details on algebraic groups and \cite{DK, PV} for invariant theory.

\begin{lemma}\label{lem2}
$$
\Phi(K[\GL(V)])\cong {\mathcal V}
$$
\end{lemma}
\begin{proof}
 Suppose that the action of $\GL(V)$ on
${\mathcal V}$ is given by
$$
\gamma:{\mathcal V}\to K[\GL(V)]\otimes {\mathcal V}.
$$
 Let $\Delta^\star:K[\GL(V)]\otimes K[\GL(V)]\to
K[\GL(V)]$ be given by $f\otimes h\mapsto fh$.
The composition
$$
(\Delta^\star\otimes \id)\circ(\id\otimes \gamma): K[\GL(V)]\otimes
{\mathcal V}\to K[\GL(V)]\otimes K[\GL(V)]\otimes {\mathcal V}\to
K[\GL(V)]\otimes {\mathcal V}
$$
is an $\GL(V)$-equivariant isomorphism
$$
K[\GL(V)]\otimes {\mathcal V}\to K[\GL(V)]\otimes {\mathcal V}.
$$
of wheeled PROPs.

 Here $\GL(V)$ acts on the left-hand $K[\GL(V)]\otimes {\mathcal V}$ by
acting on $K[\GL(V)]$ and on ${\mathcal V}$ as usual. But $\GL(V)$
acts on the right-hand $K[\GL(V)]\otimes {\mathcal V}$ by acting as
usual on $K[\GL(V)]$ and acting trivially on ${\mathcal V}$. Taking
$\GL(V)$-invariants gives us an isomorphism
$$
\Phi(K[\GL(V)])=(K[\GL(V)]\otimes {\mathcal V})^{\GL(V)}\cong
{\mathcal V}.
$$
\end{proof}
\begin{corollary}
There is a bijection between $\GL(V)$-stable subalgebras of
$K[\GL(V)]$ and subalgebras of ${\mathcal V}$.
\end{corollary}
\begin{proof}
It is easy that a $\GL(V)$-equivariant homomorphism $\psi:R\to S$ of
$K$-algebras is injective if and only if $\Phi(\psi)$ is injective.
So the corollary follows from Lemma~\ref{lem2}.
\end{proof}

\begin{prop}\label{prop2}
A wheeled PROP ${\mathcal A}$  is simple if and only
if $L:={\mathcal A}^0_0$ is a field, and the $L$-bilinear pairing
$$\langle\cdot,\cdot\rangle {\mathcal A}^p_q\times {\mathcal A}^q_p\to L
$$
is nondegenerate for all $p,q$.
\end{prop}
\begin{proof}
Suppose that ${\mathcal A}$ is simple. If $J$ is an ideal of
${\mathcal A}^0_0=L$, then $J{\mathcal A}\subseteq {\mathcal A}$ is an ideal.
So $J=0$ or $J=L$. Therefore, $L$ must be a field. Define ${\mathcal
I}=\bigoplus_{p,q\in \N}{\mathcal I}^p_q$ as follows. The space ${\mathcal I}^p_q$ is the
set of all $u\in {\mathcal I}^p_q$ such that $\langle u,\cdot\rangle:{\mathcal A}^q_p\to
L$ is the zero map. The reader may check that this defines an ideal.
Clearly ${\mathcal I}\neq {\mathcal A}$, so ${\mathcal I}=0$. This
shows that the bilinear pairing is nondegenerate.

Conversely, suppose that ${\mathcal A}^0_0=L$ is a field and that
$\langle\cdot,\cdot\rangle:{\mathcal A}^p_q\times {\mathcal A}^q_p\to L$ is nondegenerate.
Suppose that ${\mathcal I}$ is a nonzero ideal of ${\mathcal A}$.
Let $u\in {\mathcal I}^p_q$ be nonzero for some $p,q$. Then there exists a $v\in {\mathcal I}^q_p$
such that $\langle u,v\rangle=1\in {\mathcal I}^0_0$. So ${\mathcal
I}={\mathcal A}$. Hence ${\mathcal A}$ is simple.
\end{proof}

\begin{proof} [Proof of Theorem~\ref{Schrijver-generalization}]
Suppose that $G\subseteq \GL(V)$. Let $R=K[\GL(V)]^G$. Then
$${\mathcal A}:=\Phi(R)=(K[\GL(V)]^G\otimes {\mathcal V})^{\GL(V)}\cong
{\mathcal V}^G.
$$
Let $u\in A^p_q=(\V^p_q)^G$. Because $G$ is linearly reductive, there
exists a $v\in A^q_p=(\V^q_p)^G$ such that $\langle u,v\rangle=1$.
Also $A^0_0=K$. So ${\mathcal A}$ is simple by
Proposition~\ref{prop2}.
Conversely, suppose that ${\mathcal A}\subseteq {\mathcal V}$ is a
simple subalgebra. Then $R:=\Theta({\mathcal A})$ is a subalgebra of
$K[\GL(V)]$. The only $\GL(V)$-stable ideals of $R:=\Theta({\mathcal
A})$ are $0$ and $R$ itself. We claim that $R$ is finitely
generated. Indeed, let $J$ be the set of all $f\in R$ such that
$R_f$ is finitely generated, together with $0$. Then $J$ is a
nonzero, finitely generated $\GL(V)$-stable ideal. Therefore $J=R$,
so $1\in J$ and $R=R_1$ is finitely generated. So we may think of
$R$ as $K[X]$, the coordinate ring of some affine variety $X$. Now
$\GL(V)$ acts on $X$ regularly. Since $K[X]$ has no nontrivial
$\GL(V)$-invariant ideals, $X$ must be a single $\GL(V)$ orbit. So
$X=\GL(V)/G$ for some Zariski closed subgroup of $\GL(V)$. In order
for $\GL(V)/G$ to be affine, $G$ must be reductive by Matsushima's criterion (see~\cite{Matsushima60,Onishchik60,BB63}). Hence
$R=K[X]=K[\GL(V)]^G$.
\end{proof}
\begin{corollary}
Simple subalgebras of ${\mathcal V}$ are always finitely generated.
\end{corollary}
\begin{example}
$\Phi(K)={\mathcal V}^{\GL(V)}$.
\end{example}

\section{A characterization of $n$-representable wheeled PROPs}\label{sec7}
\subsection{A theorem of Procesi}
An interesting problem in the theory of polynomial identities is to determine the necessary conditions to be able to embed an algebra $R$ into $\M_n(S)$, the ring of $n \times n$ matrices over a commutative ring $S$. While this problem does not seem to have a good answer, Procesi proved a remarkable result by considering rings with trace instead. 

\begin{definition}
A $K$-algebra with trace is an algebra $R$ with a $K$-linear map $\Tr:R \rightarrow R$ satisfying

\begin{itemize}
\item $\Tr(a) b = b \Tr(a)$;
\item $\Tr(ab) = \Tr(ba)$;
\item $\Tr(\Tr(a)b) = \Tr(a)\Tr(b)$.
\end{itemize}
\end{definition}
A typical example of a trace algebra is the matrix ring $\M_n(S)$ where $S$ is a commutative $K$-algebra.
The map $\Tr:\M_n(S)\to \M_n(S)$ is given by $\Tr(A)=\trace(A)I$
where $\trace(A)\in K$ is the trace of the matrix $A\in \M_n(S)$ and $I$ is the $n\times n$ identity matrix.
If $A\in \M_n(S)$, then it has a characteristic polynomial
$$
\chi_A(T)=T^n+f_1(A)T^{n-1}+\cdots+f_{n-1}(A)T+f_n(A).
$$
The coefficients $f_1(A),\dots,f_n(A)$ can be expressed in terms of traces.
The Cayley-Hamilton identity states that $\chi_A(A)=0$.
 For example, if $n=2$ then we have
$$
\chi_A(T)=T^2-\trace(A)T+{\textstyle\frac{1}{2}}(\trace(A)^2-\trace(A^2))\in K[T]
$$
The Cayley-Hamilton for $n=2$ states that 
$$
\chi_A(A)=A^2-\trace(A)A+{\textstyle\frac{1}{2}}(\trace(A)^2-\trace(A^2))=0\in \M_{2,2}(S).
$$

\begin{theorem} [\cite{Procesi87}]
If $R$ is a $K$-algebra with trace satisfying the $n$-th Cayley-Hamilton identity, then we have an embedding $R \hookrightarrow \M_n(S)$
for some commutative $K$-algebra $S$.
\end{theorem}

We prove a similar result for wheeled PROPs. Let $V$ be a vector space of dimension $n$, and denote by $\V$, the mixed tensor algebra of $V$. For any commutative algebra $R$, we can define a wheeled PROP $$R \otimes \V = \bigoplus_{p,q \in \N} R \otimes \V^p_q.$$

It is easy to check that $R \otimes \V$ satisfies the relations $\Fund_{n+1} := \sum\limits_{\sigma \in \Sigma_{n+1}} 
\sgn(\sigma)[\sigma]$ and $\circlearrowright - n$. Note here that $\circlearrowright$ denotes the exceptional loop $\partial_1^1 (\downarrow)$. Since $\ZZ$ is the initial object, we have a unique homomorphism $\ZZ \rightarrow \PR$ for any wheeled PROP $\PR$  and hence $\Fund_{n+1}$ can be considered as an element of $\PR$. It turns out that these relations are sufficient to guarantee an embedding.

\begin{theorem} \label{embedding}
If $\PR$ is a wheeled PROP satisfying the relations $\Fund_{n+1}$ and $\circlearrowright - n$, then we have an embedding $\PR \hookrightarrow R \otimes \V$ for some commutative algebra $R$.
\end{theorem}

\subsection{Generic tensors}
 Let $V$ be an $n$-dimensional $K$-vector space with basis $e_1,e_2,\dots,e_n$, and denote the dual basis in $V^\star$ by $e^1,e^2,\dots,e^n$. We will use the short-hand notation
$$
e^{i_1,\dots,i_p}_{j_1,\dots,j_q} = e^{i_1} \otimes \cdots \otimes e^{i_p} \otimes e_{j_1} \otimes \cdots \otimes e_{j_q} \in \V^p_q.
$$

We will study relations among generic tensors. For simplicity, we illustrate the proof by looking at an special example.
 We assume that there are just $2$ generic tensors, namely a tensor $A$ of type ${2\choose 1}$, i.e., two inputs and one output, and a tensor of type ${2\choose 0}$. We write
$$ 
A = \sum\limits_{i,j,k} a_k^{ij} e_k^{ij}
$$ and 
$$ 
B = \sum\limits_{i,j} b_{i,j} e_{i,j},
$$

where $\{a_{i,j}^k\}$ and $\{b_{i,j}\}$ are indeterminates. So, we can view $A$ as an element of $R \otimes \V^2_1$ and $B$ as an element of $R\otimes \V^0_2$, where $R$ is the polynomial ring $K[\{a^{i,j}_k\},\{b_{i,j}\}]$. Let $\mathcal{W} \subset R \otimes \V$ be the sub-wheeled PROP generated by $A$ and $B$. 

Suppose we have a relation of type ${1\choose 2}$. Such a relation must be of the form $\sum_{k = 1}^r \lambda_kD_k$, where $D_1, D_2, \dots D_r$ are decorated graphs of type ${1\choose 2}$ using only $A$, $B$, and exceptional edges ({\rm id}) and exceptional loops. Since we are working in $\mathcal{W} \subset R \otimes \V$, an exceptional loop is equal to the integer $n$. Thus we can assume that there are no exceptional loops in $D_1,\dots,D_r$.

$\W$ inherits a bigrading from the polynomial ring $K[\{a^{i,j}_k\},\{b_{i,j}\}]$, where the $a$-variables have degree $(1,0)$ and the $b$-variables have degree $(0,1)$. We only have to consider relations that are homogeneous with respect to this bigrading as well. Let us assume that $D_1,\dots,D_k$ have bidegree $(1,1)$. Hence, in each of the decorated graphs $D_i$, $A$ and $B$ occur exactly once.

Let $C = \sum_{i,j,k} c_k^{i,j} e^{i,j}_k$ be a generic tensor of type ${1\choose 2}$. Let $S = K[\{a_k^{i,j}\},\{b_{i,j}\},\{c_k^{i,j}\}]$. Then $A \in R \otimes \V^2_1 \subseteq S \otimes V_1^2$, $B \in S \otimes \V^0_2$ and $C \in S\otimes V_1^2$. We have a natural pairing $\left<\cdot,\cdot,\right>:S \otimes \V^2_1 \times S \otimes V_2^1 \rightarrow S$. We have 

$$
\sum\limits_{i=1}^r \lambda_i \left<D_i,C\right> = 0.
$$

We can write $\left<D_i,C\right> = \left<A \otimes B \otimes C,\sigma_i\right>$, where $\sigma_i \in \V^4_4$, where $\sigma_i$ is just a permutation. So, we have

$$
0 = \left<A \otimes B \otimes C, \sum_{i=1}^r \lambda_i\sigma_i\right> = \sum_{i,j,k,l,p,q,r,s} a_k^{i,j}b_{l,p}c_s^{q,r} \left< e_{k,l,p,s}^{i,j,q,r}, \sum_{t =1}^r \lambda_t \sigma_t \right>.
$$

Since the monomials are linearly independent, we get 
$$
\left< e^{i,j,q,r}_{k,l,p,s}, \sum_{t = 1}^r \lambda_t \sigma_t \right> = 0
$$ 
for all $i,j,k,l,p,q,r,s$. It follows that $\sum_{t = 1}^r \lambda_t \sigma_t = 0$. Thus the relation $\sum_{t=1}^r \lambda_t D_t$ is obtained from contracting $\sum_{t = 1}^r \lambda_t \sigma_t$ with $A \otimes B$. In other words, the relation lies in ideal of a relation that does not involve $A$ and $B$. 

The above argument works for any number of generic tensors, as long as the relation involves each generic tensor at most once. Suppose we have a relation that contains $A$ $k$ times and $B$ $l$ times. Using polarization, we get a relation in generic tensors $A_1,\dots,A_k,B_1,\dots,B_l$ exactly once, where $A_i$ are of the same type as $A$ and $B_j$ are of the same type as $B$. This multilinear relation lies in the ideal of a relation that does not involve any of the $A_i$'s. Setting $A_i = A$ and $B_i = B$ shows that the original relation also in the ideal generated by a relation that does not involve $A$ or $B$.

Permutations in $\V^p_p$ are generating invariants for the action of $\GL(V)$, and all the relations among these are a consequence of $\Fund_{n+1}$. This is just a restatement of the first and second fundamental theorems of invariant theory. Hence it follows that every relation among generic tensors is a consequence of $\circlearrowright - n$ and $\Fund_{n+1}$.

\begin{proof} [Proof of Theorem~\ref{embedding}]
Given a wheeled PROP $\mathcal{A}$ satisfying $\circlearrowright - n$ and $\Fund_{n+1}$, we want to show that it is $n$-representable, i.e., $\mathcal{A} \in {\rm Obj}(\BBn)$.

Write $\mathcal{A}$ in terms of generators and relations, i.e., $\mathcal{A} = \ZZ\langle A_i \ |\  i \in I\rangle/\mathcal{I}$, where $A_i \in \mathcal{A}^{p_i}_{q_i}$. Since, the ideal $\mathcal{I}$ contains $\circlearrowright - n$ and $\Fund_{n+1}$, we see that $\mathcal{A}$ is a homomorphic image of $\mathcal{B}:= \ZZ\langle A_i \ |\  i \in I\rangle/ \left<\circlearrowright - n, \Fund_{n+1} \right>$. Since the category $\BBn$ is closed under homomorphic images, it suffices to show that $\mathcal{B}$ is in $\BBn$.

We want to get an injective map $\mathcal{B} \rightarrow R \otimes \mathcal{V}$. We take $R = K[\{a_{s_1,\dots,s_{q_i}}^{r_1,\dots,r_{p_i}}\} |\  i \in I]$, and consider $R \otimes \V$. Then we can map each $A_i \rightarrow \sum a_{s_1,\dots,s_{q_i}}^{r_1,\dots,r_{p_i}} e_{s_1,\dots,s_{q_i}}^{r_1,\dots,r_{p_i}}$ to a generic element in $R \otimes \V$. This map is well defined because both the relations $\circlearrowright-n$ and $\Fund_{n+1}$ hold in $R \otimes \V$. This map is injective by the above discussion.
\end{proof}

\noindent We will sketch how Theorem~\ref{embedding} implies Procesi's Theorem. 
\begin{proof}Suppose $(R,\Tr)$ is a trace algebra over a field $K$ of characteristic 0 satisfying the $n$-th Cayley-Hamilton identity.
Let $S$ be the subalgebra of $R$ generated by all $\Tr(a), a\in R$.
To every element $a\in S$ we introduce a generator 
$\langle a\rangle$ of type ${1\choose 1}$. Let $\GG$ be the set of all generators. In $\ZZ\langle \GG\rangle$, let ${\mathcal I}$
be the ideal generated by 
\begin{enumerate}
\item $\circlearrowright-n$;
\item  $\langle ab\rangle-\langle a\rangle\langle b\rangle$ with $a,b\in R$;
\item $\langle \Tr(a)\rangle-\partial^1_1(\langle a\rangle)\downarrow$ with $a\in R$.
\end{enumerate}
We set ${\mathcal R}=\ZZ\langle \GG\rangle/{\mathcal I}$.
Using these relations, one can show the following linear isomorphism
$$
{\mathcal R}^n_n\cong \underbrace{R\otimes_S R\otimes_S\cdots\otimes_S R}_n[\Sigma_n].
$$
Let ${\mathcal J}\subseteq {\mathcal R}$ be the ideal generated by $\Fund_{n+1}$.
Elements in $ {\mathcal J}^1_1$ are contractions of $\Fund_{n+1}\in {\mathcal R}^{n+1}_{n+1}$ with elements $a_1\otimes a_2\otimes \cdots \otimes a_n\in {\mathcal R}^n_n$. 
Such a contraction is exactly the $n$-th multi-linear Cayley-Hamilton identity for $a_1,a_2,\dots,a_n\in R$ and is therefore equal to $0$. This shows that ${\mathcal J}^1_1=0$.
Let ${\mathcal S}={\mathcal R}/{\mathcal J}$. Since ${\mathcal S}$ satisfies the relations $\Fund_{n+1}$ and $\circlearrowright-n$, we can view ${\mathcal S}$ as a sub-wheeled PROP of $T\otimes {\mathcal V}$
where $T$ is some commutative $K$-algebra and $V$ is a vector space of dimension $n$. In particular, $R$ is a subalgebra of ${\mathcal R}^1_1={\mathcal S}^1_1\subseteq T\otimes {\mathcal V}^1_1=T\otimes \End(V)$.
In other words, $R$ is a subalgebra of $\Mat_n(T)$.
\end{proof}

\subsection{Schrijver's characterization of traces of tensor representations of diagrams}
In \cite{Schrijver15}, Schrijver introduces $T$-diagrams and characterizes which functions of $T$-diagrams are traces. We will translate the main results in \cite{Schrijver15}
in terms of wheeled PROPs and reprove them in this context. 

In this subsection we assume that $K$ is algebraically closed of characteristic $0$. Let $\MM^p_q$ be the set of all monomials of type ${p\choose q}$. 
Recall that $\MM^p_q$ is a $K$-basis of $\ZZ\langle \GG\rangle^p_q$. If $A\in \MM^p_q$ and $B\in \MM^q_p$ then we can contract the outputs of $A$ with the inputs of $B$
and the inputs of $A$ with the outputs of $B$ to get a diagram $\langle A,B\rangle\in \MM^0_0$. This extends to a bilinear pairing $\ZZ\langle\GG\rangle^p_q\times \ZZ\langle \GG\rangle^q_p\to \ZZ\langle \GG\rangle^0_0$.
A function $f:\MM^0_0\to K$ is called {\em multiplicative} if $f(1)=1$ and $f(A\cdot B)=f(A)f(B)$ for all monomials $A,B\in \MM^0_0$. Here the product $A\cdot B$ is just  the disjoint union of diagrams. Note that such a function extends uniquely to
an algebra homomorphism $f:\ZZ\langle \GG\rangle^0_0\to K$.
If $A\in \MM^p_q$ (or more generally $A\in \ZZ\langle \GG\rangle^p_q$)
 then we say that $f$ annihilates $A$ if and only if $f(\langle A,B\rangle)=0$ for all $B\in \MM^q_p$ (or equivalently, for all $B\in \ZZ\langle \GG\rangle^q_p$).
Suppose that $V$ is an $n$-dimensional $K$-vector space as usual. A tensor representation of $\GG$ of dimension $n$ is a map $\rho:\GG\to \V$ such that $\rho(A)\in \V^p_q$ for all $A\in \GG$ of type ${p\choose q}$.
If $\rho$ is such a tensor representation, then it extends uniquely to a morphism of wheeled PROPs ${\rho}:\ZZ\langle \GG\rangle\to\V$ and ${\rho}^0_0:\ZZ\langle\GG\rangle^0_0\to \V^0_0=K$ restricts
to a multiplicative map $\overline{\rho}:\MM^0_0\to K$.
The following result is Schrijver's theorem in \cite{Schrijver15}.
\begin{theorem}
Suppose that $f:\MM^0_0\to K$. Then there exists a tensor representation $\rho:\GG\to K$ of dimension $\leq d$ such that $\overline{\rho}=f$ if and only if
$f$ is multiplicative and annihilates $\Fund_{d+1}$.
\end{theorem}
Let ${\mathcal I}\subseteq \ZZ\langle \GG\rangle$ be the ideal generated by $\Fund_{d+1}$.
We can reformulate this theorem using more of the wheeled PROP terminology:
\begin{theorem}
Suppose that $f:\ZZ\langle \GG\rangle^0_0\to K$ is an algebra homomorphism. Then there exists a  homomorphism
of wheeled PROPs $\rho:\ZZ\langle \GG\rangle\to {\mathcal V}$ for some vector space $V$ of dimension $n\leq d$
with $\rho^0_0=f$
if and only if $f({\mathcal I}^0_0)=0$. 
\end{theorem}
\begin{proof}
Suppose that $f=\rho^0_0$ where $\rho:\ZZ\langle\GG\rangle\to\V$ is a homomorphism of wheeled PROPs and $V$ is an $n$-dimensional vector space.
Clearly,  t $f=\rho^0_0:\ZZ\langle\GG\rangle^0_0\to K$ is an algebra homomorphism. The relation $\Fund_{n+1}$ lies in the kernel of $\rho$
and therefore, ${\mathcal I}^0_0$ is contained in the kernel of $f=\rho^0_0$.

Conversely, suppose that $f:\ZZ\langle \GG\rangle^0_0\to K$ is an algebra homomorphism and $f({\mathcal I}^0_0)=0$.
Let $t=\circlearrowright$.
We have $0=f(\langle \Fund_{d+1},\downarrow\downarrow\cdots\downarrow\rangle)=f(t(t-1)(t-2)\cdots (t-d))=f(t)f(t-1)\cdots f(t-d)$.
It follows that $f(t-n)=0$ for some $n\in \{0,1,2,\dots,d\}$.
 Let $J\subseteq \ZZ\langle \GG\rangle^0_0$ be the kernel of $f$ and let ${\mathcal J}=J\ZZ\langle \GG\rangle$ be the ideal
in $\ZZ\langle \GG\rangle$ generated by $J$. 
 Since $\Alt_{n+1}$ lies in the ideal generated by $t-n$ and $\Alt_{d+1}$, we get that $\Alt_{n+1}$ lies in ${\mathcal J}$.
Note that ${\mathcal I}^0_0\subseteq {\mathcal J}^0_0$, because $f({\mathcal I}^0_0)=0$.
Since the ideal ${\mathcal I}+{\mathcal J}$ contains $\Fund_{n+1}$ and $t-n$, there exists a finitely generated commutative $K$-algebra $R$ 
and an injective homomorphism of wheeled PROPs
$$
\psi:\ZZ\langle \GG\rangle/({\mathcal I}+{\mathcal J})\to R\otimes{\mathcal V}
$$
where $V$ is an $n$-dimensional vector space. We can choose an arbitrary maximal ideal ${\mathfrak m}$ of $R$.
Because $R$ is finitely generated, and $K$ is algebraically closed, we have $R/{\mathfrak m}\cong K$ by Hilbert's Nullstellensatz.
Now consider the composition
$$
\rho:\xymatrix{\ZZ\langle \GG\rangle\ar@{->>}[r] & \ZZ\langle \GG\rangle/({\mathcal I}+{\mathcal J})\ar@{^(->}[r]^-\psi   &R\otimes \V\ar@{->>}[r]  & R/{\mathfrak m}\otimes \V\cong \V}
$$
By construction $\rho^0_0:\ZZ\langle \GG\rangle^0_0\to {\mathcal V}^0_0=K$ is nonzero and contains ${\mathcal I}^0_0$ in its kernel.
It follows that $\rho^0_0=f$.

\end{proof}

\end{document}